\documentclass[11pt]{amsart}
\usepackage[titletoc]{appendix}

\calclayout

\def\bfomega{\boldsymbol{\omega}}
\def\bfeta{\boldsymbol{\eta}}
\def\bfb{\boldsymbol{b}}
\def\bfp{\boldsymbol{p}}

\usepackage{amssymb,amsmath,amsthm}
\usepackage{enumerate}
\usepackage{mathtools}

\usepackage[utf8]{inputenc}

\usepackage{color}
\usepackage[colorlinks,linkcolor=blue,citecolor=blue]{hyperref}

\def\AA{\mathord{\mathbf{A}}}
\def\CC{\mathord{\mathbf{C}}}
\def\RR{\mathord{\mathbf{R}}}
\def\QQ{\mathord{\mathbf{Q}}}

\def\ZZ{\mathord{\mathbf{Z}}}

\def\GG{\mathord{\mathbf{G}}}

\def\transp{^{\mathord{\textsf{T}}}}

\def\opp{{\rm op}}

\def\Sets{\textsf{Set}}
\def\Sch{\textsf{Sch}}

\def\Man{\textsf{Man}}
\def\SmVar{\textsf{SmVar}}

\def\Sym{\mathop{\rm Sym}}
\def\Sp{\mathop{\rm Sp}}
\def\GSp{\mathop{\rm GSp}}
\def\Hom{\mathop{\rm Hom}}

\def\Aut{\mathop{\rm Aut}}

\def\GL{\mathop{\rm GL}}

\def\Im{\mathop{\rm Im}}
\def\Re{\mathop{\rm Re}}
\def\Lie{\mathop{\rm Lie}}
\def\MT{\mathop{\rm MT}}

\def\Spec{\mathop{\rm Spec}}

\def\Pic{\mathop{\rm Pic }}

\def\SL{\mathop{\rm SL}}

\def\trdeg{\text{\rm trdeg}}

\def\an{\text{\rm an}}
\def\dR{\text{\rm dR}}

\def\defeq{\vcentcolon=}
\def\eqdef{=\vcentcolon}
\def\tensor{\otimes}

\def\to{\longrightarrow}
\def\mapsto{\longmapsto}

\newtheorem{theorem}{Theorem}[section]
\newtheorem{prop}[theorem]{Proposition}
\newtheorem*{conj}{Conjecture}
\newtheorem{lemma}[theorem]{Lemma}

\newtheorem{coro}[theorem]{Corollary}

\theoremstyle{definition}
\newtheorem{ex}[theorem]{Example}
\newtheorem{defi}[theorem]{Definition}
\newtheorem{obs}[theorem]{Remark}

\numberwithin{equation}{section}

\title[Higher Ramanujan equations II]{Higher Ramanujan equations II:\\ periods of abelian varieties and transcendence questions}
\author{Tiago J. Fonseca}

\address{Laboratoire de Mathématiques d'Orsay, Univ. Paris-Sud, CNRS, Université
Paris-Saclay, 91405 Orsay, France.}
\email{tiago.jardim-da-fonseca@math.u-psud.fr}

\date{\today}

\begin{document}
\maketitle

\begin{abstract}
  In the first part of this work, we have considered a moduli space $B_g$ classifying principally polarized abelian varieties of dimension $g$ endowed with a symplectic-Hodge basis, and we have constructed the higher Ramanujan vector fields $(v_{kl})_{1\le k\le l \le g}$ on it. In this second part, we study these objects from a complex analytic viewpoint. We construct a holomorphic map $\varphi_g : \mathbf{H}_g \to B_g(\CC)$, where $\mathbf{H}_g$ denotes the Siegel upper half-space of genus $g$, satisfying the system of differential equations $\frac{1}{2\pi i}\frac{\partial \varphi_g}{\partial \tau_{kl}}=v_{kl}\circ \varphi_g$, $1\le k\le l \le g$. When $g=1$, we prove that $\varphi_1$ may be identified with the triple of Eisenstein series $(E_2,E_4,E_6)$, so that the previous differential equations coincide with Ramanujan's classical relations concerning Eisenstein series. We discuss the relation between the values of $\varphi_g$ and the fields of periods of abelian varieties, and we explain how this relates to Grothendieck's periods conjecture. Finally, we prove that every leaf of the holomorphic foliation on $B_g(\CC)$ induced by the vector fields $v_{kl}$ is Zariski-dense in $B_{g,\CC}$. This last result implies a ``functional version'' of Grothendieck's periods conjecture for abelian varieties.
\end{abstract}

\tableofcontents

\section{Introduction}

\subsection{} In the first part of this work (\cite{fonseca16}) we have considered for any integer $g\ge 1$ a smooth moduli stack $\mathcal{B}_g$ over $\ZZ$ classifying principally polarized abelian varieties of dimension $g$ endowed with a \emph{symplectic-Hodge basis} of its first algebraic de Rham cohomology, and we have constructed a family $(v_{kl})_{1\le k \le l \le g}$ of $g(g+1)/2$ commuting vector fields on $\mathcal{B}_g$, the \emph{higher Ramanujan vector fields}. We have also proved that $\mathcal{B}_g\tensor \ZZ[1/2]$ is representable by a smooth quasi-projective scheme $B_g$ over $\ZZ[1/2]$. In particular, the set of complex points $B_g(\CC)$ has a natural structure of a quasi-projective complex manifold.

In this second part, we consider the differential equations defined by the higher Ramanujan vector fields, and we study their complex analytic solutions. Let
\begin{align*}
\mathbf{H}_g \defeq\{\tau \in M_{g\times g}(\CC) \mid \tau\transp = \tau\text{, }\Im \tau > 0\}
\end{align*}
be the Siegel upper half-space of genus $g$. This is a complex manifold of dimension $g(g+1)/2$ admitting the holomorphic coordinate system $(\tau_{kl})_{1\le k \le l \le g}$, where $\tau_{kl}:\mathbf{H}_g \to \CC$ is the holomorphic map associating to $\tau \in \mathbf{H}_g$ its entry in the $k$th row and $l$th column. Using the universal property of $B_g(\CC)$, we shall construct a holomorphic map
\begin{align*}
\varphi_g : \mathbf{H}_g \to B_g(\CC)
\end{align*}
satisfying the system of differential equations
\begin{align} \label{higherram}
\frac{1}{2\pi i}\frac{\partial \varphi_g}{\partial \tau_{kl}} = v_{kl}\circ \varphi_g\text{, } \ \ \ 1\le k \le l\le g\text{.}
\end{align}

When $g=1$, the Siegel upper half-space $\mathbf{H}_1$ is the Poincaré upper half-plane $\mathbf{H} = \{\tau \in \CC \mid \Im \tau >0\}$, and the classical theory of elliptic curves provides an isomorphism
\begin{align*}
B_1 \tensor \ZZ[1/6] \cong \ZZ[1/6,e_2,e_4,e_6,(e_4^3-e^2_6)^{-1}]\text{,}
\end{align*}
under which the vector field $v_{11}$ becomes the ``classical'' Ramanujan vector field (cf. \cite{fonseca16} Section 6)
\begin{align}\label{ramvectorfield}
 v\defeq \frac{e_2^2-e_4}{12}\frac{\partial}{\partial e_2} + \frac{e_2e_4-e_6}{3}\frac{\partial}{\partial e_4} + \frac{e_2e_6-e_4^2}{2}\frac{\partial}{\partial e_6}\text{.}
\end{align}
In this situation, the holomorphic map $\varphi_1$ gets identified with
\begin{align*}
\tau \mapsto (E_2(\tau),E_4(\tau),E_6(\tau))\text{,}
\end{align*}
where $E_2,E_4,E_6 : \mathbf{H} \to \CC$ are the classical (level 1) Eisenstein series normalized by $E_{2k}(+i\infty)=1$, and we obtain Ramanujan's original relations between Eisenstein series:
\begin{align*}
\frac{1}{2\pi i}\frac{dE_2}{d\tau} = \frac{E_2^2 - E_4}{12}\text{, }\ \ \frac{1}{2\pi i}\frac{dE_4}{d\tau} = \frac{E_2E_4-E_6}{3}\text{, } \ \ \frac{1}{2\pi i}\frac{dE_6}{d\tau} = \frac{E_2E_6-E^2_4}{2}\text{.}
\end{align*}

\subsection{} As explained in the introduction of \cite{fonseca16}, questions in Transcendental Number Theory constitute our main source of motivation for the study of these higher dimensional analogs of Ramanujan's equations. We shall illustrate this point by relating the values of $\varphi_g$ with Grothendieck's periods conjecture on abelian varieties. In order to fully motivate a precise statement of our result, let us digress into a discussion of periods of abelian varieties and Grothendieck's conjecture on the algebraic relations between them.

Let $X$ be a complex abelian variety and $k\subset \CC$ be the smallest algebraically closed subfield of $\CC$ over which there exists an abelian variety $X_0$ such that $X$ is isomorphic to $X_0\tensor_k\CC$ as complex abelian varieties (cf. Lemma \ref{fielddef}). By a \emph{period} of $X$, we mean any complex number of the form
\begin{align*}
\int_{\gamma}\alpha
\end{align*}
where $\alpha$ is an element of the first algebraic de Rham cohomology $H^1_{\dR}(X_0/k)$ and $\gamma \in H_1(X_0(\CC),\ZZ)$ is the class of a singular 1-cycle. We define the \emph{field of periods} $\mathcal{P}(X)$ of $X$ as the smallest subfield of $\CC$ containing $k$ and all the periods of $X$.\footnote{This definition is not standard. Usually, one starts with an abelian variety $X$ defined over a subfield $K\subset \CC$, and one defines $K$-periods in terms of $H^1_{\dR}(X/K)$. Our ``absolute'' definition considering a minimal algebraically closed field of definition is convenient for our purposes and will be justified in the sequel.} Equivalently, $\mathcal{P}(X)$ may be regarded as the field of rationality of the comparison isomorphism
\begin{align*}
H^1(X_0(\CC), \CC) = \Hom (H_1(X_0(\CC),\ZZ),\CC) \stackrel{\sim}{\to} H_{\dR}^1(X_0/k)\tensor_k \CC\text{.} 
\end{align*}
A central problem in the theory of transcendental numbers is to determine, or simply to estimate, the transcendence degree over $\mathbf{Q}$ of the field of periods $\mathcal{P}(X)$.

In a first approach, one might observe that any algebraic cycle in some power $X^n$ of $X$ induces an algebraic relation between its periods (cf. \cite{DMOS82} Proposition I.1.6). Broadly speaking, Grothendieck conjectured that \emph{every} algebraic relation between periods of an abelian variety can be ``explained'' through algebraic cycles on its powers.

A convenient way of giving a precise formulation for Grothendieck's conjecture is by means of Mumford-Tate groups. Let $X$ be a complex abelian variety, and denote by $H$ the $\QQ$-Hodge structure of weight 1 with underlying $\QQ$-vector space given by $H^1(X(\CC),\QQ)$, and Hodge filtration $F^1H$ given by $H^0(X,\Omega^1_{X/\CC})\subset H^1_{\dR}(X/\CC) \cong H^1(X(\CC),\QQ)\tensor_{\QQ}\CC$. The decomposition $H_{\CC} = F^1H \oplus \overline{F^1H}$ corresponds to the morphism of real algebraic groups
\begin{align*}
h : \CC^{\times} \to \GL(H_{\RR})\text{,}
\end{align*}
where $h(z)$ acts on $F^1H$ by a homothety of ratio $z^{-1}$, and on $\overline{F^1H}$ by a homothety of ratio $\overline{z}^{-1}$. The \emph{Mumford-Tate group} $\MT(X)$ of $X$ is defined as the smallest $\QQ$-algebraic subgroup of $\GL(H)$ such that $h$ factors through $\MT(X)_{\RR}$. It can also be interpreted as the smallest $\QQ$-algebraic subgroup of $\GL(H)\times \GG_{m,\QQ}$ fixing all Hodge classes in twisted mixed tensor powers of the $\QQ$-Hodge structure $H$ (cf. \cite{DMOS82} I.3).

The following formulation of \emph{Grothendieck's periods conjecture} (GPC) for abelian varieties is a specialization of the ``generalized Grothendieck's periods conjecture'' proposed by André (\cite{andre04} 23.4.1; see also \cite{lang66} Historical Note pp. 40-44 and \cite{grothendieck66} footnote 10).

\begin{conj}[Grothendieck-André]
For any complex abelian variety $X$, we have
\begin{align*}
\trdeg_{\QQ}\mathcal{P}(X) \stackrel{\text{?}}{\ge} \dim \MT(X)\text{.}
\end{align*} 
\end{conj}

It follows from Deligne \cite{deligne80} (cf. \cite{DMOS82} Corollary I.6.4) that we always have the upper bound
\begin{align*}
\trdeg_{\QQ}\mathcal{P}(X) \le \dim \MT(X) + \trdeg_{\QQ}k\text{,}
\end{align*}
where $k$ is the smallest algebraically closed subfield of $\CC$ over which $X$ may be defined. In particular, if $X$ is definable over the field of algebraic numbers $\overline{\QQ}$ --- the case originally considered by Grothendieck --- the above conjectural inequality becomes the conjectural equality
\begin{align*}
\trdeg_{\QQ}\mathcal{P}(X) \stackrel{\text{?}}{=} \dim \MT(X)\text{.}
\end{align*} 
In the case $g=1$, the Mumford-Tate group of a complex elliptic curve $E$ may be easily computed. Its dimension only depends on the existence or not of complex multiplication, and GPC predicts  that
\begin{align*}
\trdeg_{\QQ}\mathcal{P}(E) \stackrel{\text{?}}{\ge} \left\{\begin{array}{cl}
                                     2 & \text{if }E\text{ has complex multiplication}\\
4& \text{otherwise}\text{.}
                                    \end{array}
                              \right.
\end{align*}
Even in this minimal case, GPC is not yet established in full generality --- only the complex multiplication case is understood; see below. Nevertheless, an approach that has been proved fruitful for obtaining non-trivial lower bounds in the direction of GPC relies on a \emph{modular description} of the fields of periods of elliptic curves, which we now recall.

Let $E$ be a complex elliptic curve and let $j\in \CC$ be its $j$-invariant. Then $E$ admits a model
\begin{align*}
E: \ \ y^2 = 4x^3 -g_2x-g_3
\end{align*}
with $g_2,g_3 \in \QQ(j)$, and we can consider the algebraic differential forms defined over $\QQ(j)$
\begin{align*}
\omega \defeq \frac{dx}{y}\text{, } \ \ \ \eta \defeq x\frac{dx}{y}\text{.}
\end{align*} 
They form a (symplectic-Hodge) basis of the first algebraic de Rham cohomology $H^1_{\dR}(E/\QQ(j))$. If $(\gamma,\delta)$ is any basis of the first singular homology group $H_1(E(\CC),\ZZ)$, we may consider the periods
\begin{align*}
\omega_1 = \int_\gamma \omega\text{, }\ \ \omega_2 = \int_{\delta}\omega\text{, }\ \ \eta_1= \int_{\gamma}\eta\text{, }\ \ \eta_2=\int_{\delta}\eta\text{.}
\end{align*}
We may assume moreover that the basis $(\gamma,\delta)$ is oriented, in the sense that their topological intersection product $\gamma \cap \delta =1$.

The field of periods of $E$ is given by
\begin{align*}
\mathcal{P}(E) = \overline{\QQ(j)}(\omega_1,\omega_2,\eta_1,\eta_2)\text{,}
\end{align*}
where $\overline{\QQ(j)}$ denotes the algebraic closure of $\QQ(j)$ in $\CC$. Now, observe that $\omega_1\neq 0$ and let
\begin{align*}
\tau \defeq \frac{\omega_2}{\omega_1}\text{.}
\end{align*}
As the basis $(\gamma,\delta)$ of $H_1(E(\CC),\ZZ)$ is oriented, the complex number $\tau$ is in $\mathbf{H}$. By the classical theory of modular forms, we have
\begin{align*}
E_2(\tau) = -12 \left(\frac{\omega_1}{2\pi i} \right)\left(\frac{\eta_1}{2\pi i} \right)\text{, }\ \ E_4(\tau) = 12g_2\left(\frac{\omega_1}{2\pi i} \right)^4\text{, }\ \ E_6(\tau)= -216g_3\left(\frac{\omega_1}{2\pi i} \right)^6\text{.}
\end{align*}
Finally, Legendre's periods relation and the definition of $j$ show that $\mathcal{P}(E)$ is an algebraic extension of the field $\QQ(2\pi i, \tau, E_2(\tau),E_4(\tau),E_6(\tau))$, and we obtain
\begin{align} \label{modular}
\trdeg_{\QQ}\mathcal{P}(E) = \trdeg_{\QQ}\QQ(2\pi i, \tau, E_2(\tau),E_4(\tau),E_6(\tau))\text{.}
\end{align}

In this way, the problem of estimating the transcendence degree of fields of periods of elliptic curves translates into the problem of estimating the transcendence degree of values of some analytic functions. Accordingly, the theorem of Nesterenko \cite{nesterenko96} asserts that, for any $\tau \in \mathbf{H}$,
\begin{align*}
\trdeg_{\QQ}\QQ(e^{2\pi i\tau}, E_2(\tau),E_4(\tau),E_6(\tau)) \ge 3\text{.}
\end{align*}
As an immediate consequence, we obtain
\begin{align*}
\trdeg_{\QQ}\QQ(2\pi i, \tau, E_2(\tau),E_4(\tau),E_6(\tau)) \ge \trdeg_{\QQ} \QQ(E_2(\tau),E_4(\tau),E_6(\tau))\ge 2
\end{align*}
for any $\tau \in \mathbf{H}$. Equivalently, by equation (\ref{modular}), for any complex elliptic curve $E$, we obtain the uniform bound
\begin{align*}
\trdeg_{\QQ}\mathcal{P}(E) \ge 2\text{,}
\end{align*}
which is sharp when $E$ has complex multiplication. This last result had already been previously established by Chudnovsky (cf. \cite{chudnovsky80}) via elliptic methods.\footnote{We should also point out that the modular parameter $e^{2\pi i \tau}$, ignored in our discussion, has a geometric interpretation. Namely, it is a period of a certain $1$-motive naturally attached to $E$. We refer to \cite{bertolin02} (cf. \cite{andre04} 23.4.3) for further discussion on these matters.}

\subsection{} In this paper, we generalize the modular description (\ref{modular}). Namely, let $g\ge 1$ be an integer and, for any $\tau \in \mathbf{H}_g$, let $X_{\tau}$ be the complex abelian variety given by the (polarizable) complex torus $\CC^g/(\ZZ^g + \tau \ZZ^g)$. We shall prove that, for any $\tau \in \mathbf{H}_g$, the field of periods $\mathcal{P}(X_{\tau})$ is an algebraic extension of $\QQ(2\pi i , \tau, \varphi_g(\tau))$ (the residue field in $\AA_{\QQ}^1\times_{\QQ} \Sym_{g,\QQ}\times_{\QQ} B_{g,\QQ}$ of the complex point $(2\pi i, \tau,\varphi_g(\tau))$,  where $\Sym_g$ denotes the group scheme of symmetric matrices of order $g\times g$); in particular, we obtain 
\begin{align*}
\trdeg_{\QQ}\mathcal{P}(X_{\tau}) = \trdeg_{\QQ}\QQ(2\pi i , \tau, \varphi_g(\tau))\text{.}
\end{align*}

This generalized modular description raises the question of whether it is possible to adapt Nesterenko's methods to this higher dimensional setting. Guided by this problem, we are naturally lead to study the \emph{higher Ramanujan foliation}, namely, the holomorphic foliation on $B_g(\CC)$ generated by the higher Ramanujan vector fields.

We shall prove that every leaf of the Ramanujan foliation on $B_g(\CC)$ is Zariski-dense in $B_{g,\CC}$. This property of a foliation plays an important role, at least in the case in which leaves are one dimensional (where it implies Nesterenko's $D$-property), in the ``multiplicity estimates'' appearing in applications of differential equations to transcendental number theory (cf. \cite{binyamini14}, \cite{nesterenko89}, \cite{nesterenko96}).

The Zariski-density of the image of $\varphi_g : \mathbf{H}_g \to B_g(\CC)$ in $B_{g,\CC}$ also implies the \emph{a priori} stronger result that its graph
\begin{align*}
\{(\tau, \varphi_g(\tau)) \in {\Sym}_g(\CC)\times B_g(\CC) \mid \tau \in \mathbf{H}_g\}
\end{align*}
is Zariski-dense in $\Sym_{g,\CC} \times_{\CC} B_{g,\CC}$ (cf. \cite{BZ03} where a similar question is investigated in the context of \emph{Thetanullwerte}). This can be seen as an analog --- but not a complete generalization by the lack of the modular parameter $q$ --- of Mahler's result \cite{mahler69} on the algebraic independence of the holomorphic functions $\tau$, $e^{2\pi i \tau}$, $E_2(\tau)$, $E_4(\tau)$, and $E_6(\tau)$, of $\tau \in \mathbf{H}$. It can also be interpreted as a ``functional version'' of GPC: loosely speaking, it says that there is no algebraic relation simultaneously satisfied  by the periods of every (principally polarized) abelian variety other than the relations given by the polarization data.

The above Zariski-density results will rely on a characterization of the leaves of the higher Ramanujan foliation in terms of an action by $\Sp_{2g}(\CC)$. In fact, from the complex analytic viewpoint, the complex manifold $B_g(\CC)$ and the higher Ramanujan vector fields admit a simple description in terms of algebraic groups.

Namely, we shall explain how to realize $B_g(\CC)$ as a domain (in the analytic topology) of the quotient manifold $\Sp_{2g}(\ZZ)\backslash \Sp_{2g}(\CC)$, and we shall prove that under this identification the higher Ramanujan vector field $v_{kl}$ is induced by the left invariant holomorphic vector field on $\Sp_{2g}(\CC)$ associated to
\begin{align*}
\frac{1}{2\pi i} \left(\begin{array}{cc} 0 & \mathbf{E}^{kl} \\ 0 & 0  \end{array}\right)\in \Lie {\Sp}_{2g}(\CC)\text{,}
\end{align*}
where $\mathbf{E}^{kl}$ is the symmetric matrix of order $g\times g$ whose entry in the $k$th row and $l$th column (resp. $l$th row and $k$th column) is 1, and whose all other entries are 0.

Furthermore, the solution $\varphi_g : \mathbf{H}_g \to B_g(\CC)$ of the higher Ramanujan equations is identified to
\begin{align*}
  \tau \mapsto \left[\left(\begin{array}{cc} \mathbf{1}_g & \tau \\
                                             0 & \mathbf{1}_g
                     \end{array}\right)\right] \in {\Sp}_{2g}(\ZZ)\backslash {\Sp}_{2g}(\CC)\text{,}
\end{align*}
where $\mathbf{1}_g$ denotes the identity matrix of order $g\times g$. This enables us to obtain every leaf of the higher Ramanujan foliation as the image of a holomorphic map $\varphi_{\delta}: U_{\delta} \to B_g(\CC)$ defined on some explicitly defined open subset $U_{\delta}\subset \mathbf{H}_g$ obtained from $\varphi_g$ via a ``twist'' by some element $\delta\in \Sp_{2g}(\CC)$.

In the case $g=1$, the above twisting procedure may be illustrated as follows. Let 
\begin{align*}
\delta = \left(\begin{array}{cc}a & b \\ c& d\end{array} \right) \in {\SL}_2(\CC)\text{,}
\end{align*}
let $U_{\delta} = \{\tau \in \mathbf{H} \mid c\tau +d \neq 0\}$, and define a holomorphic map $\varphi_{\delta}:U_{\delta} \to B_1(\CC)\subset \CC^3$ by
\begin{align*}
\varphi_{\delta}(\tau) = \left((c\tau+d)^2E_2(\tau) + \frac{12c}{2\pi i}(c\tau+d), (c\tau +d)^4E_4(\tau), (c\tau+d)^6E_6(\tau)  \right)
\end{align*}
 Then one may easily check that $\varphi_{\delta}$ satisfy the differential equation
\begin{align*}
\frac{1}{2\pi i}\frac{d\varphi_{\delta}}{d\tau} = (c\tau+d)^{-2}v\circ \varphi_{\delta}
\end{align*}
where $v$ is the classical Ramanujan vector field defined by (\ref{ramvectorfield}). 

\subsection{}

As acknowledged in \cite{fonseca16}, our definition of the moduli stack $\mathcal{B}_g$ was inspired by Movasati's point of view on the Ramanujan vector field in terms of the Gauss-Manin connection on the de Rham cohomology of the universal elliptic curve (cf. \cite{movasati12} 4.2), which corresponds to the case $g=1$ of our construction.

After I completed a first version this article, H. Movasati has kindly indicated to me that a number of its results and constructions has some overlap with his article \cite{movasati13}. In this work, he considers complex analytic spaces $U$ classifying lattices in maximal totally real subspaces of some given complex vector space $V_0$ (i.e. subgroups of $V_0$ generated by a $\CC$-basis of $V_0$) satisfying suitable compatibility conditions with a fixed Hodge filtration $F^{\bullet}_0$ on $V_0$, and a fixed polarization $\psi_0$; these spaces come equipped with a natural analytic right action of the complex algebraic group
$$
G_0 = \{g\in \GL(V_0) \mid gF_0^i=F^i_0 \text{ for every }i\text{, and }g^*\psi_0=\psi_0\}\text{.}
$$
For the particular case where $V_0 = \CC^{2g}$,
$$
F_0^{\bullet} = (F^0_0=V_0\supset F^1_0=\CC^{g}\times\{0\}\supset F^2_0=0)\text{,}
$$
and $\psi_0$ is the standard (complex) symplectic form (\cite{movasati13} 5.1), the space $U$ becomes the analytic moduli space $B_g(\CC)$, investigated in the present article. Of course, the algebraic group $G_0$ coincides with our $P_g(\CC)$, and the action of $G_0$ on $U$ gets identified with the action of $P_g(\CC)$ on $B_g(\CC)$ defined in \cite{fonseca16} under $U\cong B_g(\CC)$. 

In \cite{movasati13} 3.2, Movasati also describes $U$ as a quotient $\Gamma_{\ZZ}\backslash P$, where $P$ is the space of ``period matrices'' and $\Gamma_{\ZZ}$ is some explicitly defined discrete group. In our particular case, $P$ may be identified with our $\mathbf{B}_g$ (cf. Proposition \ref{prop1}) and $\Gamma_{\ZZ}=\Sp_{2g}(\ZZ)$. Moreover, the map $\mathbf{H}_g \to P$ defined in \cite{movasati13} p. 584 coincides with our $\varphi_g: \mathbf{H}_g \to B_g(\CC)$ constructed via the universal property of $B_g(\CC)$.

In his article, Movasati explicitly states the problem of algebraizing $U$ --- i.e. finding the algebraic variety $T$ over $\overline{\QQ}$, in his notations --- and the action of $G_0$. This is solved ``by definition'' in our construction, which also shows that $T$, here called $B_{g,\overline{\QQ}}$, is smooth and quasi-projective. Movasati actually conjectures that the complex manifold $B_g(\CC)$ admit a \emph{unique} structure of complex algebraic variety (in analogy with the Baily-Borel theorem) and that it is actually quasi-affine. On his web page\footnote{See ``What is a Siegel quasi-modular form?'' in \url{http://w3.impa.br/~hossein/WikiHossein/WikiHossein.html}.}, Movasati also indicates a construction
of what we call ``higher Ramanujan vector fields'' with slightly different normalizations (cf. \cite{fonseca16} Proposition 5.7).

\subsection{Acknowledgments}

This work was supported by a public grant as part of the FMJH project, and is part of my PhD thesis under the supervision of Jean-Benoît Bost. I thank Mikolaj Fraczyk for sharing his insights on Zariski-density matters and for his interest in this work. I also thank Hossein Movasati for his kind remarks on a first version of this article, and for making me aware of his work on this circle of questions.

\subsection{Terminology and notations} \label{terminology}

Besides the terminology and notations of \cite{fonseca16}, we shall consider the following. 




\subsubsection{} Let $M$ be a complex manifold. Every holomorphic vector bundle $\pi: V \to M$ may be seen as a (commutative) relative complex Lie group over $M$. We shall occasionally identify $V$ with its corresponding locally free sheaf of $\mathcal{O}_M$-modules of holomorphic sections of $\pi$.

\subsubsection{} If $R$ is any ring, we denote the constant sheaf with values in $R$ over some complex manifold $M$ by $R_M$. A \emph{local system} of $R$-modules over $M$ is a locally constant sheaf $L$ of $R$-modules over $M$. The \emph{dual} of $L$ is denoted by  $L^{\vee} \defeq \mathcal{H}om_{R}(L,R_M)$.

The \emph{étalé space} of a local system of $R$-modules $L$ over $M$ will be denoted by $E(L)$; this is a topological covering space over $M$ whose fiber at each $p\in M$ is naturally identified to $L_p$. 

\subsubsection{} \label{notationmatrices}

Let $m,n\ge 1$ be integers. The set of matrices of order $m\times n$ over a ring  $R$ is denoted by $M_{m\times n }(R)$. We shall frequently adopt a block notation for elements in $M_{2n \times 2n}(R)$:
$$
\left(\begin{array}{cc}
A & B \\
C & D
\end{array}\right) = ( A \ B \ ; \ C \ D)\text{,}
$$
where $A,B,C,D \in M_{n\times n}(R)$.

The transpose of a matrix $M \in M_{m\times n}(R)$ is denoted by $M\transp \in M_{n\times m}(R)$. For $1\le i \le n$, $\textbf{e}_i \in M_{n\times 1}(R)$ denotes for the column vector whose entry in the $i$th line is 1, and all the others are 0. The identity matrix in $M_{n\times n}(R)$ is denoted by $\mathbf{1}_n$. For every $1\le i \le j \le n$, we denote by $\textbf{E}^{ij}$ the unique symmetric matrix $(\textbf{E}^{ij}_{kl})_{1\le k,l\le n} \in M_{n\times n}(R)$ such that
\begin{align*}
\textbf{E}^{ij}_{kl} = \begin{cases}
          1 & \text{if }  (k,l)=(i,j) \text{ or } (k,l)=(j,i)\\
          0 & \text{otherwise}\text{.}
          \end{cases}
\end{align*}

The \emph{symmetric group} $\Sym_n$ is the subgroup scheme of $M_{n\times n}$ consisting of symmetric matrices. The \emph{symplectic group} $\Sp_{2n}$ is defined as the subgroup scheme of $\GL_{2n}$ such that for every affine scheme $V=\Spec R$
\begin{align*}
{\Sp}_{2g}(V)  =  \{M \in {\GL}_{2n}(R) \mid MJM\transp =J\}
\end{align*}
where
\begin{align*}
J\defeq \left(\begin{array}{cc}
                  0 & \mathbf{1}_n \\
                  -\mathbf{1}_n& 0
                  \end{array}\right)\text{.}
\end{align*}

\begin{obs} \label{eqsympl}
As $J^2=-\mathbf{1}_{2n}$, the condition $MJM\transp = J$ is equivalent to $M^{-1} = -JM\transp J$; thus $MJM\transp = J$ if and only if $M\transp J M = J$. In particular, if we write
\begin{align*}
M = \left(\begin{array}{cc}
                  A & B \\
                  C & D
                  \end{array}\right) \in M_{2n\times 2n}(R)
\end{align*}
for some $A,B,C,D\in M_{n\times n}(R)$, then $M$ is in $\Sp_{2n}(R)$ if and only if one of the following two conditions is satisfied
\begin{enumerate}
   \item $AB\transp = BA\transp$, $CD\transp= DC\transp$, and $AD\transp -BC\transp = \mathbf{1}_n$.
   \item $A\transp C = C\transp A\text{, } B\transp D = D\transp B\text{, and } A\transp D - C\transp B = \mathbf{1}_n$.
\end{enumerate}
\end{obs}
Finally, the \emph{Siegel parabolic subgroup} $P_n$ of $\Sp_{2n}$ consists of matrices $(A \ B \ ; \ C \ D )$ in $\Sp_{2n}$ such that $C=0$. 



\subsubsection{} \label{notationresidue}

 Let $K$ be a subfield of $\CC$ and $X$ be an algebraic variety over $K$ (i.e. a reduced separated scheme of finite type over $K$). For any complex point $\overline{x}: \Spec \CC \to X$, if $x \in X$ denotes the point in the image of $\overline{x}$, and $k(x)$ denotes its residue field, we put
\begin{align*}
K(\overline{x}) \defeq k(x)\text{.}
\end{align*}
Let us remark that
\begin{align*}
\trdeg_{K}K(\overline{x}) = \min \{\dim Y \mid Y \text{ is an integral closed $K$-subscheme of }X \text{ such that }\overline{x}\in Y(\CC)\}\text{.}
\end{align*}

\section{Analytic families of complex tori, abelian varieties, and their uniformization}

In this section we briefly transpose some of the standard theory of complex tori to a relative situation, that is, we shall consider analytic families of complex tori. To both simplify and shorten our exposition, we shall assume that the parameter space is smooth (i.e. a complex manifold); this largely suffices our needs.

Most of the material included in here, and in the following section, is well known to experts --- and may be even considered as ``classical'' --- but we could not find a convenient reference in the literature.

\subsection{Relative complex tori} \label{rct}

Let $M$ be a complex manifold.

\begin{defi}
A \emph{(relative) complex torus over $M$} is a relative complex Lie group $\pi:X\to M$ over $M$ such that $\pi$ is proper with connected fibers. A morphism of complex tori over $M$ is a morphism of relative complex Lie groups over $M$.
\end{defi}

As any compact connected complex Lie group is a complex torus, every fiber of $\pi$ in the above definition is a complex torus.

In general, for any relative complex Lie group $\pi: X \to M$ over $M$, we may consider its \emph{relative Lie algebra} $\Lie_M X$; this is a holomorphic vector bundle over $M$ whose fiber at each $p\in M$ is the Lie algebra $\Lie X_p$ of the complex torus $X_p \defeq \pi^{-1}(p)$. Moreover, there exists a canonical morphism of relative complex Lie groups over $M$
\begin{align*}
\exp : {\Lie}_MX \to X
\end{align*}
restricting to the usual exponential map of complex Lie groups at each fiber.

\begin{lemma} \label{uniftorus}
Let $\pi: X \to M$ be a complex torus over $M$. Then $\exp : \Lie_MX \to X$ is a surjective submersion, and the sheaf of sections of the relative complex Lie group $\ker (\exp)$ over $M$ is canonically isomorphic to 
\begin{align*}
R_1\pi_*\ZZ_X  \defeq (R^1\pi_* \ZZ_X)^{\vee}\text{.}
\end{align*}    
\end{lemma}

This follows from the classical case where $M$ is a point via a fiber-by-fiber consideration (cf. \cite{mumford70} I.1). Note that $R_1\pi_*\ZZ_X$ is a local system of free abelian groups over $M$ whose fiber at $p\in M$ is given by the first singular homology group $H_1(X_p,\ZZ)$.

\begin{defi}
Let $V$ be a holomorphic vector bundle of rank $g$ over $M$. By a \emph{lattice} in $V$, we mean a subsheaf of abelian groups $L$ of $\mathcal{O}_M(V)$ such that
\begin{enumerate}
   \item $L$ is a local system of free abelian groups of rank $2g$,
   \item for each $p\in M$, the quotient $V_p/L_p$ is compact.
\end{enumerate}
\end{defi}

It follows from Lemma \ref{uniftorus} that, for any complex torus $\pi: X \to M$ of relative dimension $g$, $R_1\pi_*\ZZ_X$ may be canonically identified to a lattice in $\Lie_M X$.

Conversely, if $V$ is a holomorphic vector bundle of rank $g$ over $M$ and $L$ is a lattice in $V$, then the étalé space $E(L)$ of $L$ is a relative complex Lie subgroup of $V$ over $M$ and $X\defeq V/E(L)$ is a complex torus over $M$ of relative dimension $g$. Furthermore, the relative Lie algebra $\Lie_MX$ gets canonically identified with $V$ and, under this identification, $E(L)$ is the kernel of the exponential map $\exp: \Lie_MX \to X$.

\begin{obs} \label{remarktorus}
The above reasoning actually proves that the category of complex tori over $M$ of relative dimension $g$ is equivalent to the category of couples $(V,L)$ where $V$ is a holomorphic vector bundle of rank $g$ over $M$ and $L$ is a lattice in $V$; a morphism $(V,L) \to (V',L')$ in this category is given by a morphism of holomorphic vector bundles $\varphi:V \to V'$ such that $\varphi(E(L))\subset E(L')$.
\end{obs}

In what follows, we shall drop the notation $E(L)$ and identify a local system with its étalé space.

\subsection{Riemann forms and principally polarized complex tori}

Let $M$ be a complex manifold and $\pi:X \to M$ be a complex torus over $M$.

\begin{defi} \label{defiriemannform}
A \emph{Riemann form} over $X$ is a $C^{\infty}$ Hermitian metric\footnote{Our convention is that Hermitian forms are anti-linear on the first coordinate and linear on the second.} $H$ on the vector bundle  $\Lie_M X$ over $M$ such that
\begin{align*}
E \defeq \Im H
\end{align*}
takes integral values on $R_1\pi_*\ZZ_X$.
\end{defi}

Observe that $E$ is an alternating $\RR$-bilinear form. We also remark that the Hermitian metric $H$ is completely determined by $E$: for any sections $v$ and $w$ of $\Lie_M X$ we have $H(v,w)= E(v,iw) + iE(v,w)$. In particular, by abuse, we may also say that $E$ is Riemann form over $X$.

\begin{defi}
With the above notations, we say that the Riemann form $E$ is \emph{principal} if the induced morphism of local systems
\begin{align*}
   R_1\pi_*\ZZ_X &\longrightarrow (R_1\pi_*\ZZ_X)^{\vee} \cong R^1\pi_*\ZZ_X\\
           \gamma &\mapsto E( \gamma , \ )
\end{align*}
is an isomorphism.
\end{defi}

\begin{defi}
Let $M$ be a complex manifold. A \emph{principally polarized complex torus} over $M$ of relative dimension $g$ is a couple $(X,E)$, where $X$ is a complex torus over $M$ of relative dimension $g$ and $E$ is a principal Riemann form over $X$. 
\end{defi}

\begin{ex} \label{torus}
Let $g\ge 1$ and consider the Siegel upper half-space
\begin{align*}
\mathbf{H}_{g} \defeq \{\tau \in M_{g\times g}(\CC) \mid \tau = \tau^{\textsf{T}}\text{, } \Im \tau >0\}\text{.}
\end{align*}
If $g=1$, we denote $\mathbf{H} \defeq \mathbf{H}_1$; this is the Poincaré upper half-plane. Let us consider the trivial vector bundle $V\defeq \CC^g \times \mathbf{H}_g$ over $\mathbf{H}_g$ and let $L$ be the subsheaf of $\mathcal{O}_{\mathbf{H}_g}(V)$ given by the image of the morphism of sheaves of abelian groups
\begin{align*}
(\ZZ^g \oplus \ZZ^g)_{\mathbf{H}_g} &\to \mathcal{O}_{\mathbf{H}_g}(V) = \mathcal{O}_{\mathbf{H}_g}^{\oplus g}\\
            (m,n) &\mapsto m+\tau n
\end{align*}
where $m$ and $n$ are considered as column vectors of order $g$. Then $L$ is a lattice in $V$ and we denote by
\begin{align*}
\bfp_g: \mathbf{X}_g  \to \mathbf{H}_g
\end{align*}
the corresponding  complex torus over $\mathbf{H}_g$ of relative dimension $g$ (cf. Remark \ref{remarktorus}). Let $E_g$ be imaginary part of the Hermitian metric over $V$ given by
\begin{align*}
(v,w) \mapsto \overline{v}\transp (\Im \tau)^{-1}w\text{.}
\end{align*}
One may easily verify that $E_g$ takes integral values on $L$ and that $\gamma \mapsto E_g( \gamma ,  \ )$ induces an isomorphism $L \stackrel{\sim}{\to} L^{\vee}$. We thus obtain a principally polarized complex torus $(\mathbf{X}_g,E_g)$ over $\mathbf{H}_g$ of relative dimension $g$.
\end{ex}

\subsection{The category $\mathcal{T}_g$ of principally polarized complex tori of relative dimension $g$} \label{defitg}

 Let $\textsf{Man}_{/\CC}$ denote the category of complex manifolds. We define a category  $\mathcal{T}_g$ fibered in groupoids over $\Man_{/\CC}$ as follows.

\begin{enumerate}
     \item An object of the category $\mathcal{T}_g$ consists in a complex manifold $M$ and a principally polarized complex torus $(X,E)$ over $M$ of relative dimension $g$; we denote such an object by $(X,E)_{/M}$.
     \item Let $(X,E)_{/M}$ and $(X',E')_{/M'}$ be objects of $\mathcal{T}_g$. A morphism
\begin{align*}
F_{/f}:(X',E')_{/M'} \to (X,E)_{/M}
\end{align*}
in $\mathcal{T}_g$ is a cartesian diagram of complex manifolds
$$
  \raisebox{-0.5\height}{\includegraphics{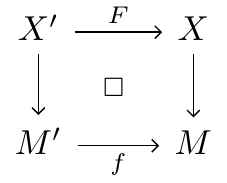}}
$$
preserving the identity sections of the complex tori and such that $E' = f^*E$ under the isomorphism of holomorphic vector bundles $\Lie_{M'}X' \stackrel{\sim}{\to} f^*\Lie_{M}X$ induced by $F$. We may also denote $(X',E') = (X,E)\times_{M}M'$.
   \item The structural functor $\mathcal{T}_g \to \Man_{/\CC}$ sends an object $(X,E)_{/M}$ of $\mathcal{T}_g$ to the complex manifold $M$, and a morphism $F_{/f}$ as above to $f$.
\end{enumerate}

\begin{ex}\label{exactionsp}
We define an action of $\Sp_{2g}(\ZZ)$ on the object $(\mathbf{X}_g,E_g)_{/\mathbf{H}_g}$ of $\mathcal{T}_g$
\begin{align*}
{\Sp}_{2g}(\ZZ) &\longrightarrow {\Aut}_{\mathcal{T}_g}\left((\mathbf{X}_g,E_g)_{/\mathbf{H}_g}\right)\\
        \gamma &\mapsto {F_\gamma}_{/f_{\gamma}}
\end{align*} 
as follows. Recall that an element $\gamma = (A \ B \ ; \ C \ D) \in \Sp_{2g}(\RR)$ acts on $\mathbf{H}_g$ by
\begin{align*}
f_{\gamma} : \mathbf{H}_g &\longrightarrow \mathbf{H}_g\\
          \tau &\mapsto \gamma\cdot \tau \defeq (A\tau + B)(C\tau+D)^{-1}\text{.}
\end{align*}
For $\gamma$ as above, consider the holomorphic map
\begin{align*}
\widetilde{F}_{\gamma}:\CC^g\times \mathbf{H}_g &\to  \CC^g\times \mathbf{H}_g\\
(z,\tau) &\mapsto  ((j(\gamma,\tau)\transp)^{-1}z,\gamma \cdot \tau)           
\end{align*}   
where 
$$
j(\gamma,\tau)\defeq C\tau+D \in {\GL}_g(\CC)\text{.}
$$
 If $\gamma \in \Sp_{2g}(\ZZ)$, then for every $\tau \in \mathbf{H}_g$ we have
\begin{align*}
\widetilde{F}_{\gamma,\tau}(\ZZ^g + \tau \ZZ^g) = \ZZ^g + (\gamma \cdot\tau) \ZZ^g\text{,}
\end{align*}
so that $\widetilde{F}_{\gamma}$ induces a holomorphic map $F_{\gamma} : \mathbf{X}_g \to \mathbf{X}_g$. One easily verifies that
$$
  \raisebox{-0.5\height}{\includegraphics{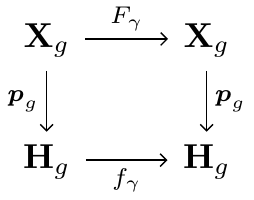}}
$$
is a cartesian diagram of complex manifolds preserving the identity sections and the Riemann forms $E_g$, i.e. it defines a morphism ${F_{\gamma}}_{/f_{\gamma}} : (\mathbf{X}_g,E_g)_{/\mathbf{H}_g} \to (\mathbf{X}_g,E_g)_{/\mathbf{H}_g}$ in $\mathcal{T}_g$. Finally, the formula 
$$
j(\gamma_1\gamma_2,\tau) = j(\gamma_1,\gamma_2\cdot \tau)j(\gamma_2,\tau)
$$
implies that ${F_{\gamma}}_{/f_{\gamma}}$ is in fact an automorphism of $(\mathbf{X}_g,E_g)_{/\mathbf{H}_g}$ in $\mathcal{T}_g$ and that $\gamma \mapsto {F_{\gamma}}_{/f_{\gamma}}$ is a morphism of groups.\footnote{Actually, it follows from Proposition \ref{reprsintsympl} below (see also Remark \ref{actionsp}) that $\gamma \mapsto {F_{\gamma}}_{/f_{\gamma}}$ is an \emph{isomorphism} of groups.}
\end{ex}

\subsection{De Rham cohomology of complex tori} 

Let $M$ be a complex manifold and $\pi : X \to M$ be a complex torus over $M$ of relative dimension $g$.

\subsubsection{}  For any integer $i\ge 0$, we define the $i$th \emph{analytic de Rham cohomology} sheaf of $\mathcal{O}_M$-modules by
\begin{align*}
\mathcal{H}^i_{\dR}(X/M) \defeq \mathbf{R}^i\pi_* \Omega^{\bullet}_{X/M}\text{,}
\end{align*}
where $\Omega^{\bullet}_{X/M}$ is the complex of relative holomorphic differential forms. If $M$ is a point, we denote $\mathcal{H}^i_{\dR}(X) \defeq \mathcal{H}^i_{\dR}(X/M)$.

\begin{obs} \label{cinftyderham}
If $M$ is a point, the analytic de Rham cohomology $\mathcal{H}^i_{\dR}(X)$ is canonically isomorphic to the quotient of the complex vector space of $C^{\infty}$ closed $i$-forms over $X$ with values in $\CC$  by the subspace of exact $i$-forms (cf. \cite{DMOS82} I.1 p. 16). 
\end{obs}

The arguments in \cite{BBM82} 2.5 prove, \emph{mutatis mutandis}, that there is a canonical isomorphism of $\mathcal{O}_M$-modules given by cup product
\begin{align*}
\bigwedge^i \mathcal{H}^1_{\dR}(X/M) \stackrel{\sim}{\to} \mathcal{H}^i_{\dR}(X/M)\text{,}
\end{align*}
and that $\mathcal{H}^1_{\dR}(X/M)$ is (the sheaf of sections of) a holomorphic vector bundle over $M$ of rank $2g$. Moreover, the canonical $\mathcal{O}_M$-morphism $\pi_*\Omega^1_{X/M} \to \mathcal{H}^1_{\dR}(X/M)$ induces an isomorphism of $\pi_*\Omega^1_{X/M}$ onto a rank $g$ subbundle of $\mathcal{H}^1_{\dR}(X/M)$ that we denote by $\mathcal{F}^1(X/M)$.

Analogously, it follows from the arguments of \cite{KO68} that $\mathcal{H}^1_{\dR}(X/M)$ is equipped with a canonical integrable holomorphic connection
\begin{align*}
\nabla : \mathcal{H}^1_{\dR}(X/M) \to \mathcal{H}^1_{\dR}(X/M) \tensor_{\mathcal{O}_{M}}\Omega^1_M \text{,} 
\end{align*}
the \emph{Gauss-Manin connection}.

Furthermore, the formation of $\mathcal{H}^1_{\dR}(X/M)$, $\mathcal{F}^1(X/M)$, and $\nabla$, are compatible with every base change in $M$.



 \subsubsection{}\label{sectioncompisom}

 There is a canonical \emph{comparison isomorphism} of holomorphic vector bundles
\begin{align} \label{compisom}
c: \mathcal{H}om_{\ZZ}(R_1\pi_*\ZZ_X,\mathcal{O}_M) \cong R^1\pi_*\ZZ_X \tensor_{\ZZ} \mathcal{O}_M  \stackrel{\sim}{\longrightarrow} \mathcal{H}^1_{\dR}(X/M)   
\end{align}
identifying the the local system of $\CC$-vector spaces $\mathcal{H}om_{\ZZ}(R_1\pi_*\ZZ_X,\CC_M) \cong R^1\pi_*\CC_X$ with the subsheaf of $\mathcal{H}^1_{\dR}(X/M)$ consisting of horizontal sections for the Gauss-Manin connection  (\cite{deligne70} I Proposition 2.28 and II 7.6-7.7). The induced pairing
\begin{align*}
\mathcal{H}^1_{\dR}(X/M) \otimes_{\ZZ}R_1\pi_*\ZZ_X  &\longrightarrow \mathcal{O}_M\\
       \alpha\tensor \gamma &\mapsto c^{-1}(\alpha)(\gamma)  \eqdef \int_{\gamma}\alpha
\end{align*}
is given at each fiber by ``integration of differential forms'' (cf. Remark \ref{cinftyderham}).
 
\begin{obs}\label{derivint}
In particular, for any section $\gamma$ of $R_1\pi_*\ZZ_X$, any $C^{\infty}$ section $\alpha$ of the vector bundle $\mathcal{H}^1_{\dR}(X/M)$, and any holomorphic vector field $\theta$ on $M$, we have
\begin{align*}
\theta \left( \int_{\gamma}\alpha \right) = \int_{\gamma}\nabla_{\theta}\alpha\text{.}
\end{align*}
\end{obs}

Recall that $R_1\pi_*\ZZ_X$ may be naturally identified with a lattice in the holomorphic vector bundle $\Lie_{M}X$. Accordingly, the dual bundle $(\Lie_MX)^{\vee}$ gets naturally identified with a holomorphic subbundle of $\mathcal{H}om_{\ZZ}(R_1\pi_*\ZZ_X,\mathcal{O}_M)$.

\begin{lemma} \label{identcomp}
With notations as above, the comparison isomorphism (\ref{compisom}) induces an isomorphism of the holomorphic vector bundle $(\Lie_MX)^{\vee}$ onto $\mathcal{F}^1(X/M)$.
\end{lemma} 

This also follows from a fiber-by-fiber argument: if $M$ is a point, by identifying $\mathcal{H}^1_{\dR}(X)$ with the $C^{\infty}$ de Rham cohomology with values in $\CC$ (Remark \ref{cinftyderham}), the subspace $\mathcal{F}^1(X)$ gets identified with the space of $(1,0)$-forms in $\mathcal{H}^1_{\dR}(X)$, and these correspond to $\Hom_{\CC}(\Lie X, \CC)$ under the de Rham isomorphism (cf. \cite{BL04} Theorem 1.4.1). 

\subsubsection{} If $X$ admits a principal Riemann form  $E$, then, by linearity, we may define a holomorphic symplectic form $\langle \ , \ \rangle_E$ on the holomorphic vector bundle $\mathcal{H}^1_{\dR}(X/M)$ over $M$ (cf. \cite{fonseca16} Appendix A) by
\begin{align*}
\langle E(\gamma, \ ) , E(\delta, \ ) \rangle_E \defeq \frac{1}{2\pi i}E(\gamma,\delta) 
\end{align*}
for any sections $\gamma$ and $\delta$ of $R_1\pi_*\ZZ_X$, where $E(\gamma , \ )$ and $E(\delta , \ )$ are regarded as sections of $\mathcal{H}^1_{\dR}(X/M)$ via the comparison isomorphism (\ref{compisom}).

Since every section of $R^1\pi_*\ZZ_X$ is horizontal for the Gauss-Manin connection $\nabla$ on $\mathcal{H}^1_{\dR}(X/M)$ under the comparison isomorphism (\ref{compisom}), the symplectic form $\langle \ , \ \rangle_{E}$ is compatible with $\nabla$: for every sections $\alpha,\beta$ of $\mathcal{H}^1_{\dR}(X/M)$, and every holomorphic vector field $\theta$ on $M$, we have
\begin{align} \label{compconn}
\theta \langle \alpha,\beta\rangle_E = \langle \nabla_{\theta}\alpha,\beta \rangle_{E} + \langle \alpha,\nabla_{\theta}\beta\rangle_E\text{.}
\end{align}
  
\subsection{Relative uniformization of complex abelian schemes} \label{unifabsch}

Let $U$ be a smooth separated $\CC$-scheme of finite type and $(X,\lambda)$ be a principally polarized abelian scheme over $U$ of relative dimension $g$. Denote by $p: X \to U$ its structural morphism. Then the associated analytic space $U^{\an}$ is a complex manifold, and the analytification $p^{\an}:X^{\an} \to U^{\an}$ of $p$ is a complex torus over $U^{\an}$ of relative dimension $g$.

Since the analytification of the coherent $\mathcal{O}_U$-module $H^1_{\dR}(X/U)$ is canonically isomorphic to $\mathcal{H}^1_{\dR}(X^{\an}/U^{\an})$, the symplectic form $\langle \ , \ \rangle_{\lambda}$ on $H^1_{\dR}(X/U)$ defined in \cite{fonseca16} 2.2 induces a symplectic form $\langle \ , \ \rangle_{\lambda}^{\an}$ on the holomorphic vector bundle $\mathcal{H}^1_{\dR}(X^{\an}/U^{\an})$ over $U^{\an}$.

\begin{lemma} \label{compsympl}
Let $\gamma$ and $\delta$ be sections of $R_1p^{\an}_*\ZZ_{X^{\an}}$, and let $\alpha$ and $\beta$ be sections of $\mathcal{H}^1_{\dR}(X^{\an}/U^{\an})$ such that $\gamma = \langle \ \ , \alpha\rangle^{\an}_{\lambda}$ and $\delta = \langle \ \ , \beta \rangle^{\an}_{\lambda}$ under (the dual of) the comparison isomorphism (\ref{compisom}). Then
\begin{enumerate}
   \item The formula
\begin{align*}
E_{\lambda}(\gamma,\delta) \defeq \frac{1}{2\pi i}\langle \alpha,\beta \rangle^{\an}_{\lambda}
\end{align*}
defines a Riemann form over $X^{\an}$.
    \item The holomorphic symplectic forms $\langle \ , \ \rangle_{E_{\lambda}}$ and $\langle \ , \ \rangle_{\lambda}^{\an}$ over $\mathcal{H}^1_{\dR}(X^{\an}/U^{\an})$ coincide.
\end{enumerate} 
\end{lemma}

\begin{proof}
We can assume $U=\Spec \CC$, so that $(X,\lambda)$ is a principally polarized complex abelian variety.  

Recall from \cite{fonseca16} 2.2 that we have constructed  an alternating bilinear form $Q_{\lambda}$ on $H^1_{\dR}(X/\CC)^{\vee}$, and that the bilinear form $\langle \ , \ \rangle_{\lambda}$ over $H^1_{\dR}(X/\CC)$ is obtained from $Q_{\lambda}$ by duality. Therefore, to prove (1), it is sufficient to prove that, under the identification of $H_1(X^{\an},\ZZ)$ with an abelian subgroup of $H^1_{\dR}(X/\CC)^{\vee}$ via (the dual of) the comparison isomorphism (\ref{compisom}), for any elements $\gamma$ and $\delta$ of $H_1(X^{\an},\ZZ)$, 
\begin{align*}
E_{\lambda}(\gamma,\delta) \defeq \frac{1}{2\pi i}  Q_{\lambda}(\gamma,\delta)
\end{align*}
is in $\ZZ$, and that the induced morphism
\begin{align*}\tag{$*$}
H_1(X^{\an},\ZZ) &\to  \Hom(H_1(X^{\an},\ZZ),\ZZ)\\
\gamma &\mapsto E_{\lambda}( \gamma , \ )
\end{align*}
is an isomorphism of abelian groups.

Note that, with this definition, (2) is automatic, since for any $\gamma,\delta \in H_1(X^{\an},\ZZ)$ we have
\begin{align*}
\langle E_{\lambda}(\gamma, \ ), E_{\lambda}(\delta, \ )\rangle_{E_{\lambda}} &= \frac{1}{2\pi i}E_{\lambda}(\gamma,\delta) = \frac{1}{(2\pi i)^2}Q_{\lambda}(\gamma,\delta) = \frac{1}{(2\pi i)^2}\langle Q_{\lambda}(\gamma, \ ),Q_{\lambda}(\delta, \ ) \rangle^{\an}_{\lambda}\\
             &= \langle \frac{1}{2\pi i}Q_{\lambda}(\gamma, \ ), \frac{1}{2\pi i}Q_{\lambda}(\delta, \ )\rangle^{\an}_{\lambda} = \langle E_{\lambda}(\gamma, \ ),E_{\lambda}(\delta, \ )\rangle^{\an}_{\lambda}\text{.}
\end{align*}
where we identified the vector space $H^1_{\dR}(X/\CC)$ with $\mathcal{H}^1_{\dR}(X^{\an})$ via the canonical analytification isomorphism. 

Now, the topological Chern class $c_{1,\text{top}}:\Pic(X) \longrightarrow H^2(X^{\an},\ZZ)$, defined via the exponential sequence
\begin{align*}
0 \longrightarrow \ZZ_{X^{\an}} \longrightarrow \mathcal{O}_{X^{\an}} &\longrightarrow \mathcal{O}_{X^{\an}}^{\times} \longrightarrow 0\\
 f&\mapsto \exp(2\pi i f)
\end{align*}
and the de Rham Chern class $c_{1,\text{dR}} : \Pic(X) \to H^2_{\dR}(X/\CC)$ (cf. \cite{fonseca16} 2.2) are related by the following commutative diagram (cf. \cite{deligne71} 2.2.5.2)
$$
  \raisebox{-0.5\height}{\includegraphics{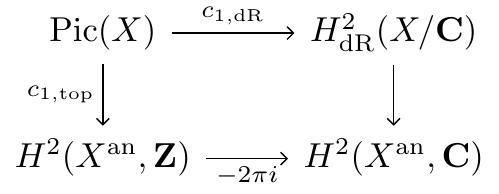}}
$$
where the arrow $H^2_{\dR}(X/\CC) \longrightarrow H^2(X^{\an},\CC) \cong \Hom(H_2(X^{\an},\ZZ),\CC)$ is given by the comparison isomorphism.

If $\mathcal{L}$ is an ample line bundle on $X$ inducing $\lambda$, then $Q_{\lambda} = c_{1,\dR}(\mathcal{L})$ under the identification $H^2_{\dR}(X/\CC)$ with the vector space of alternating bilinear forms on $H^1_{\dR}(X/\CC)^{\vee}$ (cf. \cite{fonseca16} proof of Lemma 2.1). By the commutativity of the above diagram, we see that $E_{\lambda} = -c_{1,\text{top}}(\mathcal{L})$ under the identification of $H^2(X^{\an},\ZZ)$ with the module of alternating (integral) bilinear forms on $H_1(X^{\an},\ZZ)$. This proves that $E_{\lambda}$ takes integral values.

To prove that  $(*)$ is an isomorphism, we simply use the fact that $\lambda^{\an}$ is an isomorphism of $X^{\an}$ onto its dual torus, hence the determinant of the bilinear form on $H_1(X^{\an},\ZZ)$ induced by $c_{1,\text{top}}(\mathcal{L})$ is 1 (cf. \cite{BL04} 2.4.9).
\end{proof}

Thus, for any smooth separated $\CC$-scheme of finite type $U$ and any principally polarized abelian scheme $(X,\lambda)$ over $U$ of relative dimension $g$, the above construction gives a principally polarized complex torus $(X^{\an},E_{\lambda})$ over $U^{\an}$ of relative dimension $g$.

Recall from \cite{fonseca16} 3.1 that we denote by $\mathcal{A}_{g,\CC}$ the moduli stack over $\CC$ of principally polarized abelian schemes of relative dimension $g$ over $\CC$-schemes. Let $\SmVar_{/\CC}$ be the full subcategory of $\Sch_{/\CC}$ consisting of smooth separated $\CC$-schemes of finite type, and $\mathcal{A}_{g,\CC}^{\text{sm}}$ be the full subcategory of $\mathcal{A}_{g,\CC}$ consisting of objects $(X,\lambda)_{/U}$ of $\mathcal{A}_{g,\CC}$ such that $U$ is an object of $\SmVar_{/\CC}$. 

We can summarize this paragraph by remarking that we have constructed a ``relative uniformization functor'' $\mathcal{A}_{g,\CC}^{\text{sm}} \to \mathcal{T}_g$ making the diagram
$$
  \raisebox{-0.5\height}{\includegraphics{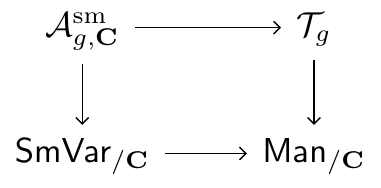}}
$$
(strictly) commutative, where  $\SmVar_{/\CC}  \to \Man_{/\CC}$ is the classical analytification functor $U\mapsto U^{\an}$.

\begin{obs} \label{algebraization}
One can prove that the above diagram is ``cartesian'' in the sense that it induces an equivalence of categories between $\mathcal{A}_{g,\CC}^{\text{sm}}$ and the full subcategory of $\mathcal{T}_g$ formed by the objects lying above the essential image of the analytification functor $\SmVar_{/\CC} \to \Man_{/\CC}$  (cf. \cite{deligne71} Rappel 4.4.3 and \cite{borel72} Theorem 3.10). In particular, for any object $U$ of $\SmVar_{/\CC}$ and any principally polarized complex torus $(X',E)$ over $U^{\an}$ of relative dimension $g$, there exists up to isomorphism a unique principally polarized abelian scheme $(X,\lambda)$ over $U$ of relative dimension $g$ such that $(X',E)_{/U^{\an}}$ is isomorphic to $(X^{\an},E_{\lambda})_{/U^{\an}}$ in $\mathcal{T}_g(U^{\an})$. In this paper, we shall only need this algebraization result when  $U=\Spec \CC$, which is classical (cf. \cite{mumford70} Corollary p. 35).
\end{obs}

\section{Analytic moduli spaces of complex abelian varieties with a symplectic-Hodge basis}

In this section we consider some moduli problems of principally polarized complex tori, regarded as functors
\begin{align*}
\mathcal{T}_g^{\opp} \to \Sets
\end{align*}
where $\mathcal{T}_g$ is the category fibered in groupoids over the category of complex manifolds $\Man_{/\CC}$ defined in \ref{defitg}.

Recall that we denote by $B_g$ the smooth quasi-projective scheme over $\ZZ[1/2]$ representing the stack $\mathcal{B}_g\tensor_{\ZZ} \ZZ[1/2]$ (see \cite{fonseca16} Theorem 4.1). We shall prove in particular that the complex manifold $B_g(\CC) = B_{g,\CC}^{\an}$ is a fine moduli space in the analytic category for principally polarized complex abelian varieties of dimension $g$ endowed with a symplectic-Hodge basis.

\subsection{Descent of principally polarized complex tori} 

Let $M$ be a complex manifold and $(X,E)$ be a principally polarized complex torus over $M$ of relative dimension $g$. 

If $M_0$ is another complex manifold and $M \to M_0$ is a holomorphic map, we say that $(X,E)$ \emph{descends} to $M_0$ if there exists a principally polarized complex torus $(X_0,E_0)$ over $M_0$ and a morphism $(X,E)_{/M} \to (X_0,E_0)_{/M_0}$ in $\mathcal{T}_g$.



\begin{lemma} \label{descentlemma}
With the above notations, suppose that there exists a proper and free left action of a discrete group $\Gamma$ on $M$. If the action of $\Gamma$ on $M$ lifts to an action of $\Gamma$ on $(X,E)_{/M}$ in the category $\mathcal{T}_g$, then $(X,E)_{/M}$ descends to a principally polarized complex torus over the quotient $\Gamma \backslash M$. 
\end{lemma} 

\begin{proof}[Sketch of the proof]
  Consider $X$ as a pair $(V,L)$, where $V$ is a holomorphic vector bundle over $M$ of rank $g$, and $L$ is a lattice in $V$ (cf. Remark \ref{remarktorus}). Then, to every $\gamma \in \Gamma$ there is associated a holomorphic map $\varphi_{\gamma} :V \to V$ making the diagram
$$
  \raisebox{-0.5\height}{\includegraphics{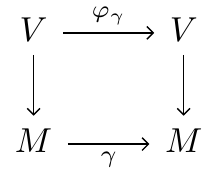}}
$$
commute, and compatible with the vector bundle structures. It follows from the commutativity of this diagram that the action of $\Gamma$ on $V$ is also proper and free. Thus, there exists a unique holomorphic vector bundle structure on the complex manifold $\Gamma \backslash V$ over $\Gamma \backslash M$ such that the canonical holomorphic map $V \to \Gamma \backslash V$ induces a vector bundle isomorphism of $V$ onto the pullback to $M$ of the vector bundle $\Gamma \backslash V$ over $\Gamma \backslash M$.

Analogously, one descends the lattice $L$ to a lattice in $\Gamma \backslash V$ (consider the étale space, for instance), and the bilinear form $E$ on $V$ to a bilinear form on $\Gamma \backslash V$, which is seen to be a principal polarization \emph{a posteriori}.
\end{proof}

\subsection{Integral symplectic bases over complex tori} 

Let $M$ be a complex manifold and $(X,E)$ be a principally polarized complex torus over $M$ of relative dimension $g$. We denote by $\pi : X \to M$ its structural morphism. 

\begin{defi}
An \emph{integral symplectic basis} of $(X,E)_{/M}$ is a trivializing $2g$-uple $(\gamma_1,\ldots,\gamma_g,\delta_1,\ldots,\delta_g)$ of global sections of $R_1\pi_*\ZZ_X$ which is symplectic with respect to the Riemann form $E$, that is,
\begin{align*}
E(\gamma_i,\gamma_j) = E(\delta_i,\delta_j)=0\ \ \text{  and } \ \ E(\gamma_i,\delta_j)=\delta_{ij}
\end{align*}
for any $1\le i\le j\le g$.
\end{defi}

\begin{ex} \label{intsymplbasis}
Consider the principally polarized complex torus $(\mathbf{X}_g,E_g)$ over $\mathbf{H}_g$ of Example \ref{torus} and recall that a section of $R_1{\bfp_g}_*\ZZ_{\mathbf{X}_g}$ is given by a column vector of holomorphic functions on $\mathbf{H}_g$ of the form $\tau \mapsto m + \tau n$, for some sections $(m,n)$ of $(\ZZ^{g}\oplus \ZZ^g)_{\mathbf{H}_g}$.  We can thus define an integral symplectic basis 
\begin{align*}
\beta_g = (\gamma_1,\ldots,\gamma_g,\delta_1,\ldots,\delta_g)
\end{align*}
of $(\mathbf{X}_g,E_g)_{/\mathbf{H}_g}$ by 
\begin{align*}
\gamma_i(\tau) \defeq \textbf{e}_i\ \ \text{ and }\ \  \delta_i(\tau)\defeq \tau \textbf{e}_i
\end{align*}
for any $\tau \in \mathbf{H}_g$.
\end{ex}

Let $(X',E')_{/M'}$ and $(X,E)_{/M}$ be objects of $\mathcal{T}_g$ with structural morphisms $\pi': X' \to M'$ and $\pi : X \to M$. If $F_{/f} : (X',E')_{/M'} \to (X,E)_{/M}$ is a morphism in $\mathcal{T}_g$, then the isomorphism of vector bundles 
\begin{align} \label{isomorphism}
{\Lie}_{M'}X' \stackrel{\sim}{\to} f^*{\Lie}_MX
\end{align}
induced by $F$ identifies the lattice $R_1\pi'_*\ZZ_{X'}$ with $f^*R_1\pi_*\ZZ_X$. If $\gamma$ is a section of $R_1\pi_*\ZZ_X$, we denote by $F^*\gamma$ the section of $R_1\pi'_*\ZZ_{X'}$ mapping to $f^*\gamma$ under (\ref{isomorphism}). As the isomorphism (\ref{isomorphism}) also preserves the corresponding Riemann forms, for any integral symplectic basis $(\gamma_1,\ldots,\gamma_g,\delta_1,\ldots,\delta_g)$ of $(X,E)_{/M}$, the $2g$-uple of global sections of $R_1\pi'_*\ZZ_{X'}$ given by
\begin{align*}
F^*\beta \defeq (F^{*}\gamma_1,\ldots,F^*\gamma_g,F^*\delta_1,\ldots,F^*\delta_g)
\end{align*}
is an integral symplectic basis of $(X',E')_{/M'}$. 

\begin{prop}[cf. \cite{BL04} Proposition 8.1.2] \label{reprsintsympl}
  The functor $\mathcal{T}_g^{\opp} \longrightarrow \mathsf{Set}$ sending an object $(X,E)_{/M}$ of $\mathcal{T}_g$ to the set of integral symplectic bases of $(X,E)_{/M}$ is representable by $(\mathbf{X}_g,E_g)_{/\mathbf{H}_g}$, with universal integral symplectic basis $\beta_g$ defined in Example \ref{intsymplbasis}.
\end{prop}

\begin{proof}
Let $(X,E)_{/M}$ be an object of $\mathcal{T}_g$ with structural morphism $\pi:X \to M$, and let $\beta = (\gamma_1,\ldots,\gamma_g,\delta_1,\ldots,\delta_g)$ be an integral symplectic basis of $(X,E)_{/M}$. Let $W$ be the real subbundle of $\Lie_MX$ generated by $\gamma_1,\ldots,\gamma_g$. Since $E$ is the imaginary part of a Hermitian metric, for any nontrivial section $\gamma$ of $W$, we have $E(\gamma,i\gamma)\neq 0$. As $W$ is isotropic with respect to $E$, it follows that $\Lie_M X = W \oplus iW$ as a real vector bundle. In particular, $\gamma \defeq (\gamma_1,\ldots,\gamma_g)$ trivializes $\Lie_M X$ as a holomorphic vector bundle. Hence, if $\delta \defeq  (\delta_1,\ldots,\delta_g)$, then there exists a unique holomorphic map $\tau : M \to \GL_g(\CC)$ such that $\delta = \gamma \tau$, where $\gamma$ and $\delta$ are regarded as row vectors of global holomorphic sections of $\Lie_M X$. 

Let $A \defeq (E(\gamma_k,i\gamma_l))_{1\le k ,l\le g} \in M_{g\times g}(\CC)$. Since 
$$
\delta = \gamma \Re \tau + i \gamma \Im \tau\text{,}
$$
the matrix of $E$ in the basis $\beta$ is given by
\begin{align*}
\left(\begin{array}{cc}
      0 & A\Im \tau\\
      -(A \Im\tau)\transp & (\Re \tau)\transp A \Im \tau - (\Im \tau)\transp A\transp \Re\tau
      \end{array} \right)\text{.}
\end{align*}
Using that $\beta$ is symplectic with respect to $E$, and that $A$ is symmetric and positive-definite (recall that $E$ is the imaginary part of a Hermitian metric), we conclude that $\tau$ factors through $\mathbf{H}_g\subset \GL_g(\CC)$.

Finally, writing $X$ as the quotient of $\Lie_M X$ by $R_1\pi_*\ZZ_X$, we see that $\tau$ lifts to a unique morphism in $\mathcal{T}_g$
\begin{align*}
F_{/\tau} : (X,E)_{/M} \to (\mathbf{X}_g,E_g)_{/\mathbf{H}_g}
\end{align*}
satisfying $F^*\beta_g = \beta$. 
\end{proof}

\begin{obs}\label{actionsp}
We may define a \emph{left} action of the group $\Sp_{2g}(\ZZ)$ on the functor $\mathcal{T}_g^{\opp} \to \Sets$ of integral symplectic bases, considered in the above proposition, as follows. Let $(X,E)_{/U}$ be an object of $\mathcal{T}_g$ and $\beta$ be an integral symplectic basis of $(X,E)_{/U}$. Let $\gamma =( A \ \ B \ ; \ C \ D) \in \Sp_{2g}(\ZZ)$, and consider $\beta = (\gamma_1, \ldots,\gamma_g,\delta_1,\ldots,\delta_g)$ as a row vector of order $2g$; then we define
\begin{align*}
\gamma \cdot \beta \defeq (\begin{array}{cccccc}\gamma_1 & \cdots & \gamma_g & \delta_1 & \cdots & \delta_g\end{array})
\left(\begin{array}{cc}
D\transp & B\transp \\
C\transp & A\transp
\end{array}\right)
\end{align*}
The morphism
\begin{align*}
{F_{\gamma}}_{/f_{\gamma}}: (\mathbf{X}_g,E_g)_{/\mathbf{H}_g} \to (\mathbf{X}_g,E_g)_{/\mathbf{H}_g}
\end{align*}
defined in Example \ref{exactionsp} is the unique morphism in $\mathcal{T}_g$ satisfying
\begin{align*}
F_{\gamma}^*\beta_g = \gamma \cdot \beta_g\text{.}
\end{align*}
\end{obs}

\subsection{Principal (symplectic) level structures} \label{level}

\subsubsection{} \label{pls}

Let $U$ be a scheme, and $X$ be an abelian scheme over $U$. Recall that, for any integer $n\ge 1$, we may define a natural pairing, the so-called \emph{Weil pairing},
\begin{align*}
X[n] \times X^t[n] \to \mu_{n,U}\text{,}
\end{align*}
where $\mu_{n,U}$ denotes the $U$-group scheme of $n$th roots of unity (cf. \cite{mumford70} IV.20).

Fix an integer $n\ge 1$, and let $\zeta_n\in \CC$ be the $n$th root of unity $e^{\frac{2\pi i}{n}}$. For any scheme $U$ over $\ZZ[1/n,\zeta_n]$, and any principally polarized abelian scheme $(X,\lambda)$ over $U$ of relative dimension $g$, by identifying $X^t[n]$ with $X[n]$ via $\lambda$, and $\mu_{n,U}$ with $(\ZZ/n\ZZ)_U$ via $\zeta_n$, we obtain a pairing
\begin{align*}
e^{\lambda}_n: X[n] \times X[n] \to (\ZZ/n\ZZ)_U\text{.}
\end{align*} 
The formation of $e^{\lambda}_n$ is compatible with every base change in $U$. Moreover, $e^{\lambda}_n$ is skew-symmetric and non-degenerate (cf. \cite{mumford70} IV.23). 

Since, for any integer $n\ge 3$, there exists a fine moduli space $A_{g,1,n}$ over $\ZZ[1/n]$ for principally polarized abelian varieties of dimension $g$ endowed with a full level $n$-structure (see \cite{GIT94} Theorem 7.9, and the following remark; see also \cite{moret-bailly85} Théorème VII.3.2), there also exists a fine moduli space $A_{g,n}$ over $\ZZ[1/n,\zeta_n]$ for principally polarized abelian varieties $(X,\lambda)$ of dimension $g$ endowed with a symplectic basis of $X[n]$ for the pairing $e^{\lambda}_n$ (cf. \cite{FC90} IV.6).  The scheme $A_{g,n}$ is quasi-projective and smooth over $\ZZ[1/n,\zeta_n]$, with connected fibers. In the sequel, we denote the universal principally polarized abelian scheme over $A_{g,n}$ by $(X_{g,n},\lambda_{g,n})$, and the universal symplectic basis of $X_{g,n}[n]$ by $\alpha_{g,n}$.

\subsubsection{}

Let $(X,E)_{/M}$ be an object of $\mathcal{T}_g$ with structural morphism $\pi:X \to M$. For any integer $n\ge 1$, by an \emph{integral symplectic basis modulo $n$} of $(X,E)_{/M}$, we mean a $2g$-uple of global sections of the local system of $\ZZ/n\ZZ$-modules 
\begin{align*}
R_1\pi_*(\ZZ/n\ZZ)_X = R_1\pi_*\ZZ_{X}/nR_1\pi_*\ZZ_{X}
\end{align*}
which is symplectic with respect to the alternating $\ZZ/n\ZZ$-linear form on $R_1\pi_*(\ZZ/n\ZZ)_X$ induced by $E$.

\begin{obs} \label{localift}
Every integral symplectic basis of $(X,E)_{/M}$ induces an integral symplectic basis modulo $n$ of $(X,E)_{/M}$. Conversely, since the natural map ${\Sp}_{2g}(\ZZ) \to {\Sp}_{2g}(\ZZ/n\ZZ)$ is surjective, locally on $M$, every integral symplectic basis modulo $n$ of $(X,E)_{/M}$ can be lifted to an integral symplectic basis of $(X,E)_{/M}$. 
\end{obs}

The notion of integral symplectic bases modulo $n$ is compatible with the notion of principal level $n$ structures of \ref{pls} in the following sense. Let $(X,\lambda)_{/U}$ be an object of $\mathcal{A}_{g,\CC}^{\text{sm}}$ (see \ref{unifabsch}) with structural morphism $p:X \to U$. The étalé space of the local system $R_1 p^{\an}_*(\ZZ/n\ZZ)_{X^{\an}}$ is canonically isomorphic to the $n$-torsion Lie subgroup $X^{\an}[n]$ of $X^{\an}$. Under this identification, the pairing $e_n^{\lambda}$ on $X[n]$ coincides, up to a sign, with the reduction modulo $n$ of the Riemann form $E_{\lambda}$ (cf. \cite{mumford70} IV.23 and IV.24), and thus an integral symplectic basis modulo $n$ of $(X^{\an},E_{\lambda})_{/U^{\an}}$ canonically corresponds  to a symplectic trivialization of $X^{\an}[n]$ with respect to $e_n^{\lambda}$.

\subsubsection{} Let $\Gamma(n)$ the kernel of the natural map ${\Sp}_{2g}(\ZZ) \to {\Sp}_{2g}(\ZZ/n\ZZ)$. Recall that for any $n\ge 3$ the induced action of $\Gamma(n)$ on $\mathbf{H}_g$ is free (\cite{mumford70} IV.21 Theorem 5) and proper.

The following proposition is well known.

\begin{prop} \label{levelstruc}
For any integer $n\ge 3$, the complex manifold $A_{g,n}(\CC)=A_{g,n,\CC}^{\an}$  is canonically biholomorphic to the quotient of $\mathbf{H}_g$ by $\Gamma(n)$, and the functor $\mathcal{T}_g^{\opp} \longrightarrow \mathsf{Set}$ sending an object $(X,E)_{/M}$ of $\mathcal{T}_g$ to the set of integral symplectic basis modulo $n$ of $(X,E)_{/M}$ is representable by $(X_{g,n,\CC}^{\an},E_{\lambda_{g,n}})_{/A_{g,n,\CC}^{\an}}$. 
\end{prop}

\begin{proof}
As the action of $\Gamma(n)$ on $\mathbf{H}_g$ is proper and free, the quotient
\begin{align*}
\mathbf{A}_{g,n} \defeq \Gamma(n)\backslash \mathbf{H}_g
\end{align*}
is a complex manifold, and the canonical holomorphic map $\mathbf{H}_g \to \mathbf{A}_{g,n}$ is a covering map with Galois group $\Gamma(n)$. Moreover, since the action of $\Gamma(n)$ on $\mathbf{H}_g$ lifts to an action of $\Gamma(n)$ on $(\mathbf{X}_g,E_g)_{/\mathbf{H}_g}$ in the category $\mathcal{T}_g$, the principally polarized complex torus $(\mathbf{X}_g,E_g)$ over $\mathbf{H}_g$ descends to a principally polarized complex torus $(\mathbf{X}_{g,n},E_{g,n})$ over $\mathbf{A}_{g,n}$ (Lemma \ref{descentlemma}).

 Let $\overline{\beta}_{g}$ be the integral symplectic basis modulo $n$ of $(\mathbf{X}_{g},E_g)_{/\mathbf{H}_g}$ obtained from $\beta_g$ by reduction modulo $n$. Then $\overline{\beta}_g$ is invariant under the action of $\Gamma(n)$, and thus it descends to an integral symplectic basis modulo $n$ of $(\mathbf{X}_{g,n},E_{g,n})_{/\mathbf{A}_{g,n}}$, say $\beta_{g,n}$. 

The object $(\mathbf{X}_{g,n},E_{g,n})_{/\mathbf{A}_{g,n}}$ of $\mathcal{T}_{g}$ so constructed represents the functor in the statement with $\beta_{g,n}$ serving as universal symplectic basis modulo $n$. Indeed, let $(X,E)_{/M}$ be an object of $\mathcal{T}_g$, and $\beta$ be an integral symplectic basis modulo $n$ of $(X,E)_{/M}$. By Remark \ref{localift}, there exists an open covering $M= \bigcup_{i\in I}U^i$ and, for each $i\in I$, an integral symplectic basis $\beta^i$ of $(X,E)_{/U^i}$ lifting $\beta$. By Proposition \ref{reprsintsympl}, we obtain for each $i\in I$ a morphism $F^i_{/f^i} : (X,E)_{/U^i} \to (\mathbf{X}_g,E_g)_{/\mathbf{H}_g}$ in $\mathcal{T}_g$ satisfying $(F^i)^*\beta_g=\beta^i$. Finally, by construction, for any $i,j \in I$, the compositions of $F^i_{/f^i}$ and $F^j_{/f^j}$ with the projection $(\mathbf{X}_g,E_g)_{/\mathbf{H}_g} \to (\mathbf{X}_{g,n},E_{g,n})_{/\mathbf{A}_{g,n}}$ agree over the intersection $U^i\cap U^j$; hence they glue to a morphism 
$$
F_{/f} : (X,E)_{/M} \to (\mathbf{X}_{g,n},E_{g,n})_{/\mathbf{A}_{g,n}}
$$ 
satisfying $F^*\beta_{g,n}=\beta$, and uniquely determined by this property.

To finish the proof, it is sufficient to show that $(X_{g,n,\CC}^{\an},E_{\lambda_{g,n}})_{/A_{g,n,\CC}^{\an}}$ is isomorphic to $(\mathbf{X}_{g,n},E_{g,n})_{/\mathbf{A}_{g,n}}$ in the category $\mathcal{T}_g$. By the compatibility of principal level $n$ structures with integral symplectic bases modulo $n$, there exists a unique morphism in $\mathcal{T}_g$
\begin{align*}
F_{/f} : (X_{g,n,\CC}^{\an},E_{\lambda_{g,n}})_{/A_{g,n,\CC}^{\an}} \to (\mathbf{X}_{g,n},E_{g,n})_{/\mathbf{A}_{g,n}}
\end{align*}
such that $F^*\beta_{g,n}$ is the integral symplectic basis modulo $n$ of $(X_{g,n,\CC}^{\an},E_{\lambda_{g,n}})_{/A_{g,n,\CC}^{\an}}$ associated to $\alpha_{g,n}$ (the universal principal level $n$ structure of $(X_{g,n},\lambda_{g,n})_{/A_{g,n}}$). Since complex tori (over a point) endowed with a principal Riemann form are algebraizable (cf. Remark \ref{algebraization}), the holomorphic map
\begin{align*}
f: A_{g,n}(\CC) = A_{g,n,\CC}^{\an} \to \mathbf{A}_{g,n}
\end{align*}
is bijective. As the complex manifolds $\mathbf{A}_{g,n}$ and $A_{g,n}(\CC)$ have same dimension, $f$ is necessarily a biholomorphism (\cite{GH78} p. 19).
\end{proof}

\subsection{Symplectic-Hodge bases over complex tori}

\subsubsection{}

Let $M$ be a complex manifold and $(X,E)$ be a principally polarized complex torus over $M$ of relative dimension $g$. As in \cite{fonseca16} Definition 2.4, by a \emph{symplectic-Hodge basis} of $(X,E)_{/M}$, we mean a $2g$-uple $b=(\omega_1,\ldots,\omega_g,\eta_1,\ldots,\eta_g)$ of global sections of the holomorphic vector bundle $\mathcal{H}^1_{\dR}(X/M)$ such that $\omega_1,\ldots,\omega_g$ are sections of the subbundle $\mathcal{F}^1(X/M)$, and $b$ is symplectic with respect to the holomorphic symplectic form $\langle \ , \ \rangle_E$.

It follows from Lemma \ref{compsympl} that this notion of symplectic-Hodge basis is compatible with its algebraic counterpart (\cite{fonseca16} Definition 2.4) via the ``relative uniformization functor'' in \ref{unifabsch}. 

\subsubsection{} \label{principalbundle}

Consider Siegel parabolic subgroup of $\Sp_{2g}(\CC)$
\begin{align*}
P_g(\CC) = 
 \left.\left\{\left(\begin{array}{cc}
                           A & B \\
                           0 & (A\transp)^{-1}
                          \end{array} \right) \in M_{2g\times 2g}(\CC) \ \right|\  A \in {\GL}_g(\CC)\text{ and } B\in M_{g\times g}(\CC)\text{ satisfy }AB\transp=BA\transp \right\}\text{.}
\end{align*}
Note that $P_g(\CC)$ is a complex Lie group of dimension $g(3g+1)/2$.

Let $(X,E)$ be a principally polarized complex torus of dimension $g$. If $b = ( \omega \  \eta )$ is a symplectic-Hodge basis of $(X,E)$, seen as a row vector of order $2g$ with coefficients in $\mathcal{H}^1_{\dR}(X)$, and $p = (A \ B \ ; \ 0 \ (A\transp)^{-1}) \in P_g(\CC)$, then we put
\begin{align*}
b \cdot p \defeq (\begin{array}{cc}\omega A & \omega B + \eta (A\transp)^{-1}  \end{array})
\end{align*}
It is easy to check that $b\cdot p$ is a symplectic-Hodge basis of $(X,E)$, and that the above formula defines a free and transitive action of $P_g(\CC)$ on the set of symplectic-Hodge bases of $(X,E)$.

\subsubsection{}
For a complex manifold $M$, let us denote by $\Man_{/M}$ the category of complex manifolds endowed with a holomorphic map to $M$.

\begin{lemma}[cf. \cite{fonseca16} Corollary 3.4] \label{relrepr}
Let $M$ be a complex manifold and $(X,E)$ be a principally polarized complex torus over $M$ of relative dimension $g$. The functor
\begin{align*}
\mathsf{Man}_{/M}^{\opp} &\to \mathsf{Set}\\
           M' & \mapsto \{\text{symplectic-Hodge bases of }(X,E)\times_M M'\}
\end{align*}
is representable by a principal $P_g(\CC)$-bundle $B(X,E)$ over $M$.
\end{lemma}

\begin{proof}
Let us denote by $\pi: V \to M$ the holomorphic vector bundle $\mathcal{H}^1_{\dR}(X/M)^{\oplus g}$ over $M$. For any $p\in M$, the fiber $\pi^{-1}(p)=V_p$ is the vector space of $g$-uples $(\alpha_1,\ldots,\alpha_g)$, with each $\alpha_i \in \mathcal{H}^1_{\dR}(X_p)$. Let $B$ be the locally closed analytic subspace of $V$ consisting of points $v=(\alpha_1,\ldots,\alpha_g)$ of $V$ such that 
\begin{align*}
L \defeq \CC \alpha_1 + \cdots + \CC \alpha_g
\end{align*}
is a Lagrangian subspace of $\mathcal{H}^1_{\dR}(X_{\pi(v)})$ with respect to $\langle \ , \ \rangle_{E_{\pi(v)}}$ satisfying
\begin{align*}
\mathcal{F}^1(X_{\pi(v)}) \oplus L = \mathcal{H}^1_{\dR}(X_{\pi(v)})\text{.}
\end{align*}

By \cite{fonseca16} Proposition A.7. (2), a symplectic-Hodge basis $(\omega_1,\ldots,\omega_g,\eta_1,\ldots,\eta_g)$ of a principally polarized complex torus is uniquely determined by $(\eta_1,\ldots,\eta_g)$. In particular, for each $p\in M$, the fiber $B_p=B\cap V_p$ may be naturally identified with the set of symplectic-Hodge bases of $(X_p,E_p)$.

Thus, it follows from \ref{principalbundle} that $B$ is a principal $P_g(\CC)$-bundle over $M$; in particular, it is a complex manifold. We also conclude from the above paragraph that $B$ represents the functor in the statement.
\end{proof}

\begin{obs}\label{relreprcomp}
The above construction is compatible, under analytification, with its algebraic counterpart. Namely, let $U$ be a smooth separated $\CC$-scheme of finite type, and $(X,\lambda)$ be a principally polarized abelian scheme over $U$. The complex manifold $B(X^{\an},E_{\lambda})$ over $U^{\an}$ constructed in Lemma \ref{relrepr} is canonically isomorphic to the analytification of the scheme $B(X,\lambda)$ over $U$ constructed in \cite{fonseca16} Corollary 3.4.
\end{obs}

Recall that we denote by $(X_g,\lambda_g)$ the universal principally polarized abelian scheme over $B_g$, and by $b_g$ the universal symplectic-Hodge basis of $(X_g,\lambda_g)_{/B_g}$.

\begin{prop} \label{reprsympl}
 The functor $\mathcal{T}_g^{\opp} \longrightarrow \mathsf{Set}$ sending an object $(X,E)_{/M}$ of $\mathcal{T}_g$ to the set of symplectic-Hodge bases of $(X,E)_{/M}$ is representable by $(X_{g,\CC}^{\an},E_{\lambda_g})_{/B_{g,\CC}^{\an}}$, with universal symplectic-Hodge basis $b_g$.
 \end{prop}

\begin{proof}
By Lemma \ref{relrepr}, there exists a complex manifold $\mathbf{B}_g\defeq B(\mathbf{X}_g,E_g)$ over $\mathbf{H}_g$ representing the functor
\begin{align*}
\Man_{/\mathbf{H}_g}^{\opp} &\to \Sets\\
                    M &\mapsto \{\text{symplectic-Hodge bases of }(\mathbf{X}_g,E_{g})\times_{\mathbf{H}_g}M\}
\end{align*}
Let $(\mathbf{X}_{\mathbf{B}_g}, E_{\mathbf{B}_g}) = (\mathbf{X}_g,E_g)\times_{\mathbf{H}_g}\mathbf{B}_g$. Note that the principally polarized complex torus $(\mathbf{X}_{\mathbf{B}_g}, E_{\mathbf{B}_g})$ over $\mathbf{B}_g$ is equipped with a universal symplectic-Hodge basis $b_{\mathbf{B}_g}$, and with an integral symplectic basis $\beta_{\mathbf{B}_g}$ obtained by pullback from $\beta_g$ via the canonical morphism $(\mathbf{X}_{\mathbf{B}_g}, E_{\mathbf{B}_g})_{/\mathbf{B}_g} \to (\mathbf{X}_g,E_g)_{/\mathbf{H}_g}$ in $\mathcal{T}_g$.

We now remark that $(\mathbf{X}_{\mathbf{B}_g}, E_{\mathbf{B}_g})_{/\mathbf{B}_g}$ represents the functor $\mathcal{T}_g^{\opp} \to \Sets$ sending an object $(X,E)_{/M}$ of $\mathcal{T}_g$ to the cartesian product of the set of symplectic-Hodge bases of $(X,E)_{/M}$ with the set of integral symplectic bases of $(X,E)_{/M}$, with $(b_{\mathbf{B}_g},\beta_{\mathbf{B}_g})$ serving as a universal object. Thus, for any element $\gamma\in \Sp_{2g}(\ZZ)$, there exists a unique automorphism ${\Psi_{\gamma}}_{/\psi_{\gamma}}$ of $(\mathbf{X}_{\mathbf{B}_g}, E_{\mathbf{B}_g})_{/\mathbf{B}_g}$ in $\mathcal{T}_g$ such that $\Psi_{\gamma}^*b_{\mathbf{B}_g} =b_{\mathbf{B}_g}$ and $\Psi_{\gamma}^*\beta_{\mathbf{B}_g}=\gamma\cdot \beta_{\mathbf{B}_g}$ (where the left action of $\Sp_{2g}(\ZZ)$ on integral symplectic bases is defined as in Remark \ref{actionsp}). 

As the functor $\underline{B}_g:\mathcal{A}_g^{\opp} \to \Sets$ is rigid over $\CC$ (\cite{fonseca16} Lemma 4.3), we see that
\begin{enumerate}
  \item $\gamma \mapsto {\Psi_{\gamma}}_{/\psi_{\gamma}}$ is in fact an action of $\Sp_{2g}(\ZZ)$ on $(\mathbf{X}_{\mathbf{B}_g}, E_{\mathbf{B}_g})_{/\mathbf{B}_g}$ in the category $\mathcal{T}_g$, and
  \item the action $\gamma \mapsto \psi_{\gamma}$ of $\Sp_{2g}(\ZZ)$ on the complex manifold $\mathbf{B}_g$ is free; it is also proper since it lifts the action on $\mathbf{H}_g$.
\end{enumerate} 

Let $M$ be the quotient manifold $\Sp_{2g}(\ZZ)\backslash \mathbf{B}_g$ and descend $(\mathbf{X}_{\mathbf{B}_g},E_{\mathbf{B}_g})$ to a principally polarized complex torus $(X,E)$ over $M$. Since $b_{\mathbf{B}_g}$ is invariant under the action of $\Sp_{2g}(\ZZ)$, we can descend it to a symplectic-Hodge basis $b$ of $(X,E)_{/M}$. As in the proof of Proposition \ref{levelstruc}, we may check that $(X,E)_{/M}$ represents the functor in the statement, with $b$ serving as universal symplectic-Hodge basis. 

To finish the proof, we must prove that $(X,E)_{/M}$ is isomorphic to $(X_{g,\CC}^{\an},E_{\lambda_g})_{/B_{g,\CC}^{\an}}$ in $\mathcal{T}_g$. For this, it is sufficient to remark that, by the universal property of $(X,E)_{/M}$, there exists a unique morphism in $\mathcal{T}_g$
$$
F_{/f}:(X_{g,\CC}^{\an},E_{\lambda_g})_{/B_{g,\CC}^{\an}} \to (X,E)_{/M}
$$
satisfying $F^*b = b_g$, and that the holomorphic map
\begin{align*}
f: B_g(\CC)= B_{g,\CC}^{\an} \to M
\end{align*}
is bijective since principally polarized complex tori (over a point) are algebraizable (cf. Remark \ref{algebraization}); then $f$ is necessarily a biholomorphism (\cite{GH78} p. 19).
\end{proof}


\section{The higher Ramanujan equations and their analytic solution $\varphi_g$}

Fix an integer $g\ge 1$. Let us consider the holomorphic coordinate system $(\tau_{kl})_{1\le k \le l \le g}$ on the complex manifold $\mathbf{H}_g$, where $\tau_{kl}: \mathbf{H}_g \to \CC$ associates to any $\tau \in \mathbf{H}_g$ its entry in the $k$th row and $l$th column. To this system of coordinates is attached a family $(\theta_{kl})_{1\le k \le l \le g}$ of holomorphic vector fields on $\mathbf{H}_g$, defined by
\begin{align*}
\theta_{kl} \defeq \frac{1}{2\pi i}\frac{\partial}{\partial\tau_{kl}}\text{.} 
\end{align*}

Let $(v_{kl})_{1\le k \le l \le g}$ be the family of holomorphic vector fields on $B_g(\CC)$ induced by the higher Ramanujan vector fields on $\mathcal{B}_g$ defined in \cite{fonseca16} 5.3.

\begin{defi} \label{defihreq}
Let $U$ be an open subset of $\mathbf{H}_g$. We say that a holomorphic map $u: U \to B_g(\CC)$ is a solution of the \emph{higher Ramanujan equations} if
\begin{align*}
\theta_{kl}u = v_{kl}\circ u
\end{align*}
for every $1\le k \le l \le g$.
\end{defi}

 In this section, we construct a global holomorphic solution
\begin{align*}
\varphi_g: \mathbf{H}_g \to B_g(\CC)
\end{align*}
of the higher Ramanujan equations. In view of the universal property of the moduli space $B_g(\CC)$ (Proposition \ref{reprsympl}), the holomorphic map $\varphi_g$ will be induced by a certain symplectic-Hodge basis of the principally polarized complex torus $(\mathbf{X}_g,E_g)$ over $\mathbf{H}_g$.


\subsection{Definition of $\varphi_g$ and statement of our main theorem} \label{defphig}

Recall that the comparison isomorphism (\ref{compisom}) identifies the holomorphic vector bundle $(\Lie_{\mathbf{H}_g} \mathbf{X}_g)^{\vee}$ over $\mathbf{H}_g$ with $\mathcal{F}^1(\mathbf{X}_g/\mathbf{H}_g)$ (Lemma \ref{identcomp}). Moreover, it follows from the construction of $\mathbf{X}_g$ in Example \ref{torus} that $\Lie_{\mathbf{H}_g}\mathbf{X}_g$ is canonically isomorphic to the trivial vector bundle $\CC^g\times \mathbf{H}_g$ over $\mathbf{H}_g$. Under this isomorphism, we define the holomorphic frame
\begin{align*}
(dz_1,\ldots,dz_g)
\end{align*}
of $\mathcal{F}^1(\mathbf{X}_g/\mathbf{H}_g)$ as the dual of the canonical holomorphic frame of $\CC^g\times \mathbf{H}_g$.

\begin{theorem}\label{theoremsolution}
For each $1\le k \le g$, consider the global sections of $\mathcal{H}^1_{\dR}(\mathbf{X}_g/\mathbf{H}_g)$
\begin{align*}
\bfomega_k \defeq 2\pi i\, dz_k\text{, } \ \ \ \bfeta_k \defeq \nabla_{\theta_{kk}}\bfomega_k
\end{align*}
where $\nabla$ denotes the Gauss-Manin connection on $\mathcal{H}^1_{\dR}(\mathbf{X}_g/\mathbf{H}_g)$. Then,
\begin{enumerate}
   \item The $2g$-uple 
\begin{align*}
\bfb_g \defeq (\bfomega_1,\ldots,\bfomega_g,\bfeta_1,\ldots,\bfeta_g)
\end{align*}
of holomorphic global sections of $\mathcal{H}^1_{\dR}(\mathbf{X}_g/\mathbf{H}_g)$ is a symplectic-Hodge basis of the principally polarized complex torus $(\mathbf{X}_g,E_g)_{/\mathbf{H}_g}$.
   \item The holomorphic map  
\begin{align*}
\varphi_g: \mathbf{H}_g\to B_g(\CC)
\end{align*}
   corresponding to $\bfb_g$ by the universal property of $B_g(\CC)$ is a solution of the higher Ramanujan equations (Definition \ref{defihreq}).
\end{enumerate}
\end{theorem}

The main idea in our proof is to compute with a $C^{\infty}$ trivialization of the vector bundle $\mathcal{H}^1_{\dR}(\mathbf{X}_g/\mathbf{H}_g)$. In the next subsection we develop some preliminary background.

\subsection{Preliminary results}

Consider the \emph{complex conjugation}, seen as a $C^{\infty}$ morphism of real vector bundles over $\mathbf{H}_g$,
\begin{align*}
\mathcal{H}^1_{\dR}(\mathbf{X}_g/\mathbf{H}_g) &\to \mathcal{H}^1_{\dR}(\mathbf{X}_g/\mathbf{H}_g)\\
                               \alpha &\mapsto \overline{\alpha}
\end{align*}
induced by the comparison isomorphism (\ref{compisom}), and denote $d\bar{z}_k \defeq \overline{dz_k}$ for every $1\le k \le g$. We may check fiber by fiber that the $2g$-uple of $C^{\infty}$ global sections of $\mathcal{H}^1_{\dR}(\mathbf{X}_g/\mathbf{H}_g)$
\begin{align*}
(dz_1,\ldots,dz_g,d\bar{z}_1,\ldots,d\bar{z}_g)
\end{align*}
trivializes $\mathcal{H}^1_{\dR}(\mathbf{X}_g/\mathbf{H}_g)$ as a $C^{\infty}$ complex vector bundle over $\mathbf{H}_g$.

For $1\le i \le j \le g$ and $1\le k \le g$, let us define
\begin{align*}
\bfeta^{ij}_k \defeq \nabla_{\theta_{ij}} \bfomega_k\text{,}
\end{align*}
so that
\begin{align*}
\bfeta_k = \bfeta_k^{kk}\text{.}
\end{align*}


\begin{prop} \label{lemme1}
Consider the notations in \ref{notationmatrices}. For every $1\le i \le j \le g$ and $1\le k \le g$, we have
\begin{align*}
\bfeta^{ij}_k = \sum_{l=1}^g \mathbf{e}^{\mathsf{T}}_k  \mathbf{E}^{ij}(\Im \tau)^{-1}\mathbf{e}_l \Im dz_l
\end{align*} 
as a $C^{\infty}$ section of $\mathcal{H}^1_{\dR}(\mathbf{X}_g/\mathbf{H}_g)$, where $\Im dz_l \defeq (dz_l -d\bar{z}_l)/2i$. 
\end{prop}

\begin{proof}
For $1\le i \le j \le g$ and $1\le k,l \le g$, let $\lambda_{kl}^{ij}$ and $\mu_{kl}^{ij}$ be the $C^{\infty}$ functions on $\mathbf{H}_g$ with values in $\CC$  defined by the equation 
\begin{align*}
\bfeta_{k}^{ij} = \sum_{l=1}^g (\lambda_{kl}^{ij} dz_l + \mu_{kl}^{ij}d\bar{z}_l)\text{.}
\end{align*}
 We must prove that $\lambda_{kl}^{ij} + \mu_{kl}^{ij}=0$ and that $\lambda^{ij}_{kl} = \frac{1}{2i}\textbf{e}\transp_k\textbf{E}^{ij}(\Im \tau)^{-1}\textbf{e}_l$.

Let us consider the integral symplectic basis $\beta_g=(\gamma_1,\ldots,\gamma_g,\delta_1,\ldots,\delta_g)$ of $R_1{\bfp_g}_*\ZZ_{\mathbf{X}_g}$ defined in Example \ref{intsymplbasis}. For every $1\le i \le j \le g$ and $1\le k,l\le g$, we have (cf. Remark \ref{derivint})
\begin{align*}
\int_{\gamma_l}\bfeta^{ij}_k=\int_{\gamma_l}\nabla_{\frac{\partial}{\partial \tau_{ij}}}dz_k = \frac{\partial}{\partial \tau_{ij}}\int_{\gamma_l}dz_k = \frac{\partial}{\partial \tau_{ij}}\delta_{kl} = 0
\end{align*}
and
\begin{align*}
\int_{\delta_l} \bfeta_{k}^{ij} = \int_{\delta_l}\nabla_{\frac{\partial}{\partial \tau_{ij}}}dz_k =\frac{\partial}{\partial \tau_{ij}}\int_{\delta_l}dz_k = \frac{\partial}{\partial \tau_{ij}}\tau_{kl} = \textbf{E}^{ij}_{kl}\text{.}
\end{align*}
Thus, by definition of $\lambda^{ij}_{kl}$ and $\mu^{ij}_{kl}$, we obtain
\begin{align*}
0=\int_{\gamma_l}\bfeta^{ij}_k =\sum_{m=1}^g\left( \lambda^{ij}_{km}\int_{\gamma_l}dz_m + \mu^{ij}_{km}\int_{\gamma_l}d\bar{z}_m\right) = \lambda^{ij}_{kl} + \mu^{ij}_{kl}
\end{align*}
and
\begin{align*}
\textbf{E}^{ij}_{kl} = \int_{\delta_l}\bfeta^{ij}_k = \sum_{m=1}^g\left( \lambda^{ij}_{km}\int_{\delta_l}dz_m + \mu^{ij}_{km}\int_{\delta_l}d\bar{z}_m\right) = \sum_{m=1}^g \lambda^{ij}_{km}(\tau_{ml}-\overline{\tau_{ml}}) = 2i\sum_{m=1}^g \lambda^{ij}_{km}(\Im\tau)_{ml}\text{.}
\end{align*}
In matricial notation, if we put $\lambda^{ij} \defeq (\lambda^{ij}_{kl})_{1\le k,l\le g} \in M_{g\times g}(\CC)$, then we have shown that
\begin{align*}
2i\lambda^{ij}\Im \tau = \textbf{E}^{ij}
\end{align*}
The assertion follows.
\end{proof}

Specializing to the case $i=j=k$ in the above proposition, we obtain the following formulas.

\begin{coro} \label{caraceta}
For any $1\le k \le g$, we have
\begin{align*}
\bfeta_k = \sum_{l=1}^g ((\Im \tau)^{-1})_{kl} \Im dz_l\text{.}
\end{align*}
In particular, $\bfeta_k$ is the unique global section of $\mathcal{H}^1_{\dR}(\mathbf{X}_g/\mathbf{H}_g)$ satisfying
\begin{align*}
\int_{\gamma_l}\bfeta_k =0\ \ \ \text{ and } \ \ \ \int_{\delta_l}\bfeta_k = \delta_{kl}
\end{align*}
for every $1\le l \le g$. In other words, $\bfeta_k$ may be identified with $E_g(\gamma_k, \ )$ under the comparison isomorphism (\ref{compisom}).
\end{coro}

Since every section of $R^1{\bfp_g}_*\ZZ_{\mathbf{X}_g} = (R_1{\bfp_g}_*\ZZ_{\mathbf{X}_g})^{\vee}$, seen as a section of $\mathcal{H}^1_{\dR}(\mathbf{X}_g/\mathbf{H}_g)$ via the comparison isomorphism (\ref{compisom}), is horizontal for the Gauss-Manin connection, we obtain the next corollary.


\begin{coro}\label{corohor}
For any $1\le k \le g$, the global section $\bfeta_k$ of $\mathcal{H}^1_{\dR}(\mathbf{X}_g/\mathbf{H}_g)$ is horizontal for the Gauss-Manin connection:
\begin{align*}
\nabla \bfeta_k = 0\text{.}
\end{align*}
\end{coro}

Our next goal is to use the duality given by the Riemann form $E_g$ to express $dz_l$ in terms of $C^{\infty}$ sections of $\Lie_{\mathbf{H}_g}\mathbf{X}_g$.

\begin{lemma} \label{lemme2}
 Let $1\le k \le g$, and denote by $\tau_k$ the $k$-th column of $\tau \in \mathbf{H}_g$. Then
\begin{align*}
dz_k = -E_g(i\Im\tau_k, \ ) + iE_g(\Im \tau_k, \ )
\end{align*}
as a $C^{\infty}$ section of $\mathcal{H}^1_{\dR}(\mathbf{X}_g/\mathbf{H}_g)$ under the comparison isomorphism (\ref{compisom}).
\end{lemma}


\begin{proof}
Note that $\Im\tau_k = (\Im\tau) \textbf{e}_k$. Let $\gamma$ be a section of $R_1{\bfp_g}_*\ZZ_{\mathbf{X}_g}$. As $\Im \tau$ is symmetric and $\gamma = \Re \gamma + i \Im \gamma$, we have
\begin{align*}
-E_g(i\Im\tau_k, \gamma ) + iE_g(\Im\tau_k ,\gamma) &= -\Im (\overline{i\Im \tau_k}\transp (\Im \tau)^{-1} \gamma) + i\Im(\overline{\Im \tau_k}\transp (\Im \tau)^{-1}\gamma)\\
&= \Im (i\textbf{e}_k\transp (\Im \tau) (\Im\tau)^{-1}\gamma) + i\Im (\textbf{e}_k\transp (\Im \tau) (\Im\tau)^{-1}\gamma) \\ 
             &=\Re (\textbf{e}_k\transp \gamma) + i\Im (\textbf{e}\transp_k \gamma) \\
             &= \textbf{e}_k\transp \gamma = dz_k(\gamma) \text{.}
\end{align*}
\end{proof}

\subsection{Proof of Theorem \ref{theoremsolution}}

We prove parts (1) and (2) separately.

\begin{proof}[Proof of Theorem \ref{theoremsolution} (1)]
As each $\bfomega_k$ is by definition a section of $\mathcal{F}^1(\mathbf{X}_g/\mathbf{H}_g)$, to prove that $\bfb_g$ is a symplectic-Hodge basis of $(\mathbf{X}_g,E_g)_{/\mathbf{H}_g}$ it is sufficient to show that it is a symplectic trivialization of $\mathcal{H}^1_{\dR}(\mathbf{X}_g/\mathbf{H}_g)$ with respect to the holomorphic symplectic form $\langle \ , \ \rangle_{E_g}$. For this, we claim that it is enough to prove that
\begin{align} \label{equation1} \tag{$*$}
\langle \bfomega_i,\bfeta_j \rangle_{E_g} = \delta_{ij}
\end{align}
for every $1\le i \le j\le g$. Indeed, by Corollary \ref{corohor} and by the compatibility (\ref{compconn}), equation (\ref{equation1}) implies that $\langle \bfeta_i,\bfeta_j\rangle_{E_g}=0$ (apply $\nabla_{\theta_{ii}}$). Since we already know that $\mathcal{F}^1(\mathbf{X}_g/\mathbf{H}_g)$ is Lagrangian, this proves indeed that $\bfb_g$ is a symplectic trivialization of $\mathcal{H}^1_{\dR}(\mathbf{X}_g/\mathbf{H}_g)$.

Fix $1\le i \le j \le g$. By Corollary \ref{caraceta}, we have 
\begin{align*}
 \bfeta_j = \sum_{l=1}^g ((\Im \tau)^{-1})_{jl}\Im dz_l \text{,}
 \end{align*}
thus
 \begin{align*}
 \langle \bfomega_i,\bfeta_j \rangle_{E_g} = 2\pi i\sum_{l=1}^g ((\Im \tau)^{-1})_{jl} \langle dz_i,\Im dz_l \rangle_{E_g}\text{.} 
 \end{align*}
Now, using Lemma \ref{lemme2}, we obtain
 \begin{align*}
 \langle dz_i,\Im dz_l \rangle_{E_g} &= \langle -E_g(i\Im\tau_i, \ ) + iE_g(\Im \tau_i, \ ), E_g( \Im\tau_l, \ )\rangle_{E_g}\\
         &=-\langle E_g(i\Im \tau_i, \ ), E_g( \Im \tau_l,\ )\rangle_{E_g} + i\langle E_g(\Im \tau_i , \ ),E_g( \Im\tau_l , \ )\rangle_{E_g}\\
&=\frac{1}{2\pi i}\left(-E_g(i\Im \tau_i,\Im\tau_l) + iE_g(\Im \tau_i,\Im \tau_l) \right)\\
&=\frac{1}{2\pi i}\Im (i\Im \tau_i\transp(\Im\tau)^{-1} \Im\tau_l)\\
&=\frac{1}{2\pi i}\textbf{e}_i\transp (\Im \tau)\textbf{e}_l = \frac{1}{2\pi i}(\Im \tau)_{il}\text{.}
\end{align*}
Therefore, since $\Im \tau$ is symmetric,
\begin{align*}
 \langle \bfomega_i,\bfeta_j \rangle_{E_g} = \sum_{l=1}^g((\Im \tau)^{-1})_{jl}(\Im \tau)_{li}  = \delta_{ij}\text{.}
\end{align*}
\end{proof}

Let $M$ be a complex manifold and $(X,E)$ be a principally polarized complex torus over $M$ of relative dimension $g$. Let us denote by $\nabla$ the Gauss-Manin connection on $\mathcal{H}^1_{\dR}(X/M)$. To any symplectic-Hodge basis $b=(\omega_1,\ldots,\omega_g,\eta_1,\ldots,\eta_g)$ of $(X,E)_{/M}$ we can associate a morphism of $\mathcal{O}_M$-modules (cf. \cite{fonseca16} Theorem 5.4)
\begin{align*}
 c: T_M &\longrightarrow \Gamma^2(\mathcal{F}^1(X/M)^{\vee})\oplus \mathcal{H}^1_{\dR}(X/M)^{\oplus g}\\
 \theta &\mapsto (\kappa(\theta), \nabla_{\theta}\eta_1,\ldots, \nabla_{\theta}\eta_g)
\end{align*} 
where
\begin{align*}
\kappa : T_M &\longrightarrow \Gamma^2(\mathcal{F}^1(X/M)^{\vee})\\
          \theta &\mapsto \sum_{i=1}^g  \langle \ \ , \eta_i \rangle_E \tensor \langle \ \ , \nabla_{\theta}\omega_i \rangle_E 
\end{align*}
is the \emph{Kodaira-Spencer morphism} defined as in \cite{fonseca16} 5.1.2.\footnote{Recall that $\Gamma^2(\mathcal{F}^1(X/M)^{\vee})$ denotes the submodule of symmetric tensors in $\mathcal{F}^1(X/M)^{\vee}\tensor \mathcal{F}^1(X/M)^{\vee}$.}

This construction is compatible with base change: if $M'$ is another complex manifold, $(X',E')$ is a principally polarized complex torus over $M'$ of relative dimension $g$, and $b'$ is a symplectic-Hodge basis of $(X',E')_{/M'}$, then for any morphism $F_{/f} : (X',E')_{/M'} \to (X,E)_{/M}$ in $\mathcal{T}_g$ such that $F^*b=b'$, the diagram
$$
  \raisebox{-0.5\height}{\includegraphics{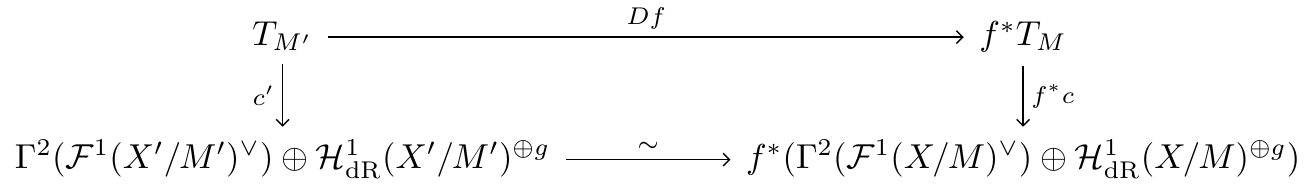}}
$$
commutes.

\begin{obs}
Applying the above construction to the universal symplectic-Hodge basis $b_g$ of $(X_{g,\CC}^{\an},E_{\lambda_{g}})_{/B_{g,\CC}^{\an}}$, we obtain the analytification $c_{g,\CC}^{\an}$ of the morphism $c_g$ defined in \cite{fonseca16} Theorem 5.4.
\end{obs}

Part (2) in Theorem \ref{theoremsolution} will be an easy consequence of the following characterization.

\begin{prop} \label{equivalences}
Let $U\subset \mathbf{H}_g$ be an open subset and $u: U \to B_g(\CC)$ be the holomorphic map corresponding to a principally polarized complex torus $(X,E)$ over $U$ endowed with some symplectic-Hodge basis $b=(\omega_1,\ldots,\omega_g,\eta_1,\ldots,\eta_g)$. Then the following are equivalent:
\begin{enumerate}
    \item $u$ is a solution of the higher Ramanujan equations.
    \item For every $1\le i \le j \le g$, we have
\begin{align*}
c(\theta_{ij}) = u^*c_{g,\CC}^{\an}(v_{ij})
\end{align*}
where $c : T_U \longrightarrow \Gamma^2(\mathcal{F}^1(X/U)^{\vee})\oplus \mathcal{H}^1_{\dR}(X/U)^{\oplus g}$ is the morphism defined above for the symplectic-Hodge basis $b$ of $(X,E)_{/U}$, and $v_{ij}$ are the higher Ramanujan vector fields on $B_g(\CC)$.
    \item For every $1\le i \le j \le g$, we have
      \begin{itemize}
    \item[($i$)] $\nabla_{\theta_{ij}} \omega_i = \eta_j$, $\nabla_{\theta_{ij}} \omega_j = \eta_i$, and $\nabla_{\theta_{ij}}\omega_k =0$, for $k\notin \{i,j\}$
    \item[($ii$)] $\nabla_{\theta_{ij}}\eta_k =0$, for $1 \le k \le g$.
 \end{itemize}
\end{enumerate} 
\end{prop} 

\begin{proof}
  The equivalence between (1) and (2) follows from the commutativity of the diagram
  $$
  \raisebox{-0.5\height}{\includegraphics{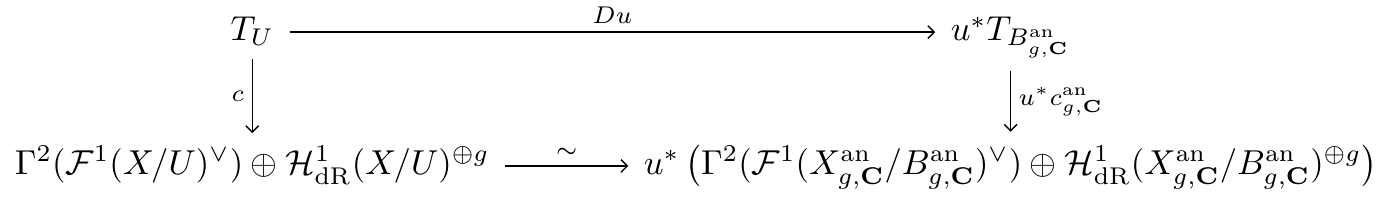}}
$$
and the injectivity of $u^*c^{\an}_{g,\CC}$ (cf. \cite{fonseca16} Theorem 5.4). 

The same argument in the proof of \cite{fonseca16} Proposition 5.7 proves the equivalence between (2) and (3).
\end{proof}

\begin{proof}[Proof of Theorem \ref{theoremsolution} (2)]
By Proposition \ref{equivalences}, it is sufficient to prove that, for every $1\le i \le j \le g$, we have
\begin{itemize}
     \item[($i$)] $\nabla_{\theta_{ij}} \bfomega_i = \bfeta_j$, $\nabla_{\theta_{ij}} \bfomega_j = \bfeta_i$, and $\nabla_{\theta_{ij}}\bfomega_k =0$, for $k\notin \{i,j\}$
     \item[($ii$)] $\nabla_{\theta_{ij}}\bfeta_k =0$, for $1 \le k \le g$.
\end{itemize}
  Now, ($i$) follows directly from Proposition \ref{lemme1}, and ($ii$) is the content of Corollary \ref{corohor}.
\end{proof}


\subsection{The case $g=1$}

 We now explicitly describe the holomorphic map $\varphi_1 : \mathbf{H} \to B_1(\CC)$. For every $k\ge 1$, let us denote by $E_{2k} : \mathbf{H} \to \CC$ the classical (level 1) Eisenstein series of weight $2k$ normalized by $E_{2k}(+i\infty) = 1$.

\begin{prop} \label{solutionexpl}
  Let $\varphi_1 : \mathbf{H} \to B_1(\CC)$ be the holomorphic map defined in Theorem \ref{theoremsolution} for $g=1$. Then
 \begin{enumerate}
    \item Under the identification $B_1 \cong \Spec \ZZ[1/2,b_2,b_4,b_6,\Delta^{-1}]$ of \cite{fonseca16} Theorem 6.2, we have
 \begin{align*}
 \varphi_1(\tau) = \left(E_2(\tau),\frac{1}{2}\theta E_2(\tau),\frac{1}{6}\theta^2 E_2(\tau) \right)\text{,}
 \end{align*}
where $\theta = \frac{1}{2\pi i}\frac{d}{d\tau}$.
    \item Under the identification of \cite{fonseca16} Remark 6.3, we have
 \begin{align*}
 \varphi_1(\tau) = (E_2(\tau),E_4(\tau),E_6(\tau))\text{.}
 \end{align*} 
 \end{enumerate}
 \end{prop}

 \begin{proof}
By the change-of-coordinates formulas in \cite{fonseca16} Remark 6.3, it is sufficient to prove (2). For every $\tau \in \mathbf{H}$, we denote by $X_{\tau}$ the complex elliptic curve defined by the equation
\begin{align*}
 y^2 = 4x^3 - \frac{E_4(\tau)}{12}x + \frac{E_6(\tau)}{216}.
\end{align*}
Recall that there is an isomorphism
\begin{align*}
F_{\tau} : \mathbf{X}_{1,\tau} &\stackrel{\sim}{\longrightarrow} X_{\tau}(\CC)\\
                z &\mapsto \begin{cases}
                            \left(\left(\frac{1}{2\pi i}\right)^2\wp_{\tau}(z) : \left(\frac{1}{2\pi i}\right)^3\wp'_{\tau}(z):1\right) &\text{ if }   z\neq 0\\
                            (0:1:0) &\text{ if } z=0
                           \end{cases}
\end{align*}
where $\wp_{\tau}$ denotes the Weierstrass $\wp$-function associated to the lattice $\ZZ + \tau \ZZ \subset \CC$. Furthermore, we have (cf. \cite{katz73} A1.3.16)
\begin{align*}
\bfeta_1 = \nabla_{\theta}\bfomega_1 = \frac{1}{2\pi i}\wp_{\tau}(z)dz - \frac{E_2(\tau)}{12}\bfomega_1\text{.}
\end{align*}
Let $\varphi : \mathbf{H} \to B_1(\CC)$ be given by $\varphi(\tau)=(E_2(\tau),E_4(\tau),E_6(\tau))$. Then, for every $\tau \in \mathbf{H}$, the isomorphism
\begin{align*}
\Phi_{\tau} : \mathbf{X}_{1,\tau} &\stackrel{\sim}{\longrightarrow} X_{1,\varphi(\tau)}(\CC)\\ 
                z &\mapsto \begin{cases}
                            \left(\left(\frac{1}{2\pi i}\right)^2\wp_{\tau}(z) - \frac{E_2(\tau)}{12} : \left(\frac{1}{2\pi i}\right)^3\wp'_{\tau}(z):1\right) &\text{ if } z\neq 0\\
                            (0:1:0) &\text{ if } z=0
                           \end{cases}
\end{align*}
satisfies
\begin{align*}
\Phi_{\tau}^*\left(\frac{dx}{y}\right) = \bfomega_1\ \ \text{ and }\ \ \Phi^*_{\tau}\left(x\frac{dx}{y}\right) =\bfeta_1\text{.}
\end{align*}
By making $\tau$ vary in $\mathbf{H}$, we obtain a morphism $\Phi_{/\varphi}: \mathbf{X}_1 \to X^{\an}_{1,\CC}$ in $\mathcal{T}_1$. Since $\Phi^*(dx/y,xdx/y) = \bfb_1$, we must have $\varphi=\varphi_1$ by definition of $\varphi_1$.
\end{proof}


\section{Values of $\varphi_g$ and transcendence degree of fields of periods of abelian varieties} \label{sectionvaluesphi}

Let $X$ be a complex abelian variety (resp. a complex torus). For any subfield $k$ of $\CC$, we say that \emph{$X$ is definable over $k$} if there exists an abelian variety $X_0$ over $k$ such that $X$ is isomorphic to $X_0\tensor_k \CC$ as a complex abelian variety (resp. isomorphic to $X_0(\CC)$ as a complex torus).

\begin{lemma}\label{fielddef}
For any complex abelian variety $X$ (resp. polarizable complex torus), there exists a smallest algebraically closed subfield $k$ of $\CC$ over which $X$ is definable.
\end{lemma}
\begin{proof}
Let $g$ be the dimension of $X$, and let $\lambda:X \to X^t$ be any polarization on $X$ (not necessarily principal). If $\lambda$ is of degree $d^2$, then the isomorphism class of the couple $(X,\lambda)$ defines a complex point $\overline{x}\in A_{g,d,1}(\CC)$, where $A_{g,d,1}$ denotes the coarse moduli space over $\QQ$ of abelian varieties of dimension $g$ endowed with a polarization of degree $d^2$ (cf. \cite{GIT94} Theorem 7.10). Let $k$ be the algebraic closure in $\CC$ of the residue field $\mathbf{Q}(\overline{x})$ (see \ref{notationresidue}). 

It is clear that $X$ is definable over $k$. To prove that $k$ is the smallest algebraically closed subfield of $\CC$ with this property, let $k'$ be any algebraically closed subfield of $\CC$ over which there is an abelian variety $X'$ such that $X$ is isomorphic to $X'\tensor_{k'}\CC$. As $k'$ is algebraically closed, the polarization $\lambda$ on $X$ descends to a polarization $\lambda'$ on $X'$ such that $(X,\lambda)$ and $(X',\lambda')\tensor_{k'}\CC$ are isomorphic as polarized complex abelian varieties\footnote{This follows from the fact that, for any abelian varieties $X$ and $Y$ over a field $K$, the functor $\Sch_{/K}^{\opp} \longrightarrow \Sets$ given by $U \mapsto {\Hom}_{\textsf{GpSch}_{/U}}(X\times_{K}U,Y\times_{K}U)$ is representable by an étale $K$-scheme.}. Thus the morphism $\overline{x} : \Spec \CC \to A_{g,d,1}$ factors through $\Spec k'$, which implies that $\QQ(\overline{x})\subset k'$. As $k'$ is algebraically closed, we obtain $k\subset k'$.
\end{proof}

\begin{defi}
Let $X$ be a complex abelian variety, $k$ be the smallest algebraically closed subfield of $\CC$ over which $X$ is definable, and fix a $k$-model $X_0$ of $X$. The \emph{field of periods} $\mathcal{P}(X)$ of $X$ is defined as the smallest subfield of $\CC$ containing $k$ and the image of pairing
\begin{align*}
H^1_{\dR}(X_0/k)\tensor H_1(X_0(\CC),\ZZ)   &\to \CC\\
           \alpha \tensor \gamma &\mapsto \int_{\gamma}\alpha
\end{align*} 
given by ``integration of differential forms'' (cf. \ref{sectioncompisom}).
\end{defi}

Note that $\mathcal{P}(X)$ does not depend on the choice of $X_0$.

This section is devoted to the proof of the following theorem.

\begin{theorem} \label{trdeg}
Let $g\ge 1$ be an integer. With notations as in Example \ref{torus} and Theorem \ref{theoremsolution}, for every $\tau \in \mathbf{H}_g$ the field of periods $\mathcal{P}(\mathbf{X}_{g,\tau})$ of the polarizable complex torus $\mathbf{X}_{g,\tau}\defeq \CC^g/(\ZZ^g + \tau \ZZ^g)$  is an algebraic extension of $\QQ(2\pi i,\tau,\varphi_g(\tau))$. In particular,
\begin{align*}
\trdeg_{\QQ}\QQ(2\pi i,\tau,\varphi_g(\tau)) = \trdeg_{\QQ} \mathcal{P}(\mathbf{X}_{g,\tau})\text{.}
\end{align*}
\end{theorem}  

\subsection{Period matrices}

Let us consider the \emph{general symplectic group}; namely, the subgroup scheme  $\GSp_{2g}$ of $\GL_{2g}$ over $\Spec \ZZ$ such that, for every affine scheme $V=\Spec R$, we have
\begin{align*}
{\GSp}_{2g}(V) = \left.\left\{\left(\begin{array}{cc}
                           A & B \\
                           C & D
                          \end{array} \right) \in M_{2g\times 2g}(R) \ \right| \begin{aligned} &\ \ \ \ \ \ \ \ \ \ \ \   A,B,C,D \in M_{g\times g}(R) \text{ satisfy  } \\ &AB\transp = BA\transp\text{, }  CD\transp= DC\transp\text{, and } AD\transp -BC\transp \in R^{\times} \mathbf{1}_g \end{aligned}\right\}\text{.}
\end{align*}
We can define a morphism of group schemes 
$$
\nu: {\GSp}_{2g} \to \mathbf{G}_m
$$ 
as follows: if $s=( A \ B \ ; \ C \ D  ) \in \GSp_{2g}(V)$, then $\nu(s) \in R^{\times}$ satisfies $AD\transp - BC\transp = \nu(s) \textbf{1}_g$. Note that $\Sp_{2g}$ is the kernel of $\nu$. 

We denote by $\GSp_{2g}^*$ the open subscheme of $\GSp_{2g}$ defined by the condition $A \in \GL_g(R)$ in the above notations.

Let $(X,E)$ be a principally polarized complex torus of dimension $g$, and $b=(\omega_1,\ldots,\omega_g,\eta_1,\ldots,\eta_g)$ (resp. $\beta =  (\gamma_1,\ldots,\gamma_g,\delta_1,\ldots,\delta_g)$) be a symplectic-Hodge basis (resp. an integral symplectic basis) of $(X,E)$.

\begin{defi} \label{defiperiodmatrix}
The \emph{period matrix} of $(X,E)$ with respect to $b$ and $\beta$ is defined by
\begin{align*}
P(X,E,b,\beta) \defeq \left(\begin{array}{cc}
                               \Omega_1 & N_1 \\
                               \Omega_2 & N_2
                               \end{array}\right) \in M_{2g\times 2g}(\CC)\text{,}
\end{align*}
where
\begin{align*}
(\Omega_1)_{ij} \defeq \int_{\gamma_i}\omega_j \ \ & \ \ (N_1)_{ij} \defeq \int_{\gamma_i}\eta_j\\
(\Omega_2)_{ij} \defeq \int_{\delta_i}\omega_j \ \ & \ \ (N_2)_{ij} \defeq \int_{\delta_i}\eta_j\text{.}
\end{align*}
\end{defi} 

 Note that $P(X,E,b,\beta)$ is simply the matrix of the comparison isomorphism (\ref{compisom}) with respect to the bases $b$ of $\mathcal{H}^1_{\dR}(X)$ and $(E( \ \ ,\delta_1),\ldots,E(\ \ ,\delta_g),E( \gamma_1 , \ \ ),\ldots, E(\gamma_g, \ \ ))$ of $\Hom(H_1(X,\ZZ),\CC)$. 

\begin{obs}\label{remarkperiodmatrix}
In particular, let $(X,\lambda)$ be a principally polarized complex abelian variety, $k$ be the smallest algebraically closed subfield of $\CC$ over which $X$ is definable, and $(X_0,\lambda_0)$ be a $k$-model of $(X,\lambda)$. Then, if $b$ is any symplectic-Hodge basis of $(X_0,\lambda_0)$, and $\beta$ is any integral symplectic basis of $(X^{\an},E_{\lambda})$, the field of periods $\mathcal{P}(X)$ of $X$ is generated over $k$ by the coefficients of the period matrix $P(X^{\an},E_{\lambda},b,\beta)$. 
\end{obs} 

\begin{lemma}\label{lemmaperiodmatrix}
For any $(X,E,b,\beta)$ as above, we have
\begin{enumerate}
    \item $P(X,E,b,\beta) \in \GSp_{2g}(\CC)$ and $\nu(P(X,E,b,\beta)) =2\pi i$, 
    \item $\Omega^1 \in \GL_g(\CC)$ (i.e. $P(X,E,b,\beta) \in \GSp_{2g}^*(\CC)$) and $\Omega_2\Omega_1^{-1} \in \mathbf{H}_g$.
\end{enumerate}
\end{lemma}

\begin{proof}
Knowing that $P(X,E,b,\beta)$ is a base change matrix with respect to symplectic bases, (1) is simply a rephrasing of Lemma \ref{compsympl} and (2) is a particular case of the classical \emph{Riemann relations} (cf. proof of Proposition \ref{reprsintsympl}). 
\end{proof}


\subsection{Auxiliary lemmas}

We shall need the following auxiliary results. 

\begin{lemma}\label{isomorphismgsp}
The morphism of schemes
\begin{align*}
{\GSp}_{2g}^* &\to \mathbf{G}_m \times_{\ZZ} {\Sym}_g \times_{\ZZ} P_g\\
         s &\mapsto (\nu(s),\tau(s),p(s))
\end{align*}
where
\begin{align*}
\tau\left(\begin{array}{cc}
          A & B \\
          C & D \end{array}\right) \defeq CA^{-1}\ \ \ \text{ and }\ \ \  p\left(\begin{array}{cc}
          A & B \\
          C & D \end{array}\right) \defeq \left(\begin{array}{cc}
          A^{-1} & -B\transp \\
          0 & A\transp \end{array}\right)
\end{align*}
is an isomorphism.
\end{lemma}

\begin{proof}
We simply remark that
\begin{align*}
\left(\lambda, Z,\left(\begin{array}{cc}
          X & Y \\
          0 & (X\transp)^{-1} \end{array}\right)  \right)\mapsto \left(\begin{array}{cc}
          X^{-1} & -Y\transp \\
          ZX^{-1} & (\lambda \mathbf{1}_g  - ZX^{-1}Y)X\transp \end{array}\right) 
\end{align*}
is an inverse to the morphism defined in the statement.
\end{proof}

\begin{lemma} \label{computationallemma}
Let $F: (X,E) \to (X',E')$ be an isomorphism of principally polarized complex tori of dimension $g$, $\beta= (\gamma_1,\ldots,\gamma_g,\delta_1,\ldots,\delta_g)$  be an integral symplectic basis of $(X,E)$ and $b'$ be a symplectic-Hodge basis of $(X',E')$. We denote by $F_*\beta$ the integral symplectic basis of $(X',E')$ given by pushforward in singular homology. Then the symplectic-Hodge basis
\begin{align*}
b = (\omega_1,\ldots,\omega_g,\eta_1,\ldots,\eta_g) \defeq F^*b' \cdot p\left( \frac{1}{2\pi i}P(X',E',b',F_*\beta)\right)
\end{align*} 
of $(X,E)$ satisfy
\begin{align*}
\int_{\gamma_i}\eta_j = 0 \text{, }\int_{\delta_i}\eta_j = \delta_{ij}
\end{align*}
for every $1\le i,j\le g$.
\end{lemma}

The proof this lemma is a straightforward computation. 

\subsection{Proof of Theorem \ref{trdeg}} \label{proofthmtrdeg}

Let $A_g$ be the coarse moduli space associated to the Deligne-Mumford stack $\mathcal{A}_g \to \Spec \ZZ$ (which exists as an algebraic space by the Keel-Mori theorem, cf. \cite{olsson16} Theorem 11.1.2). We recall that $A_g$ is a quasi-projective scheme over $\Spec \ZZ$ (cf. \cite{moret-bailly85} VII Théorème 4.2) endowed with a canonical morphism $\mathcal{A}_g \to A_g$ inducing, for every algebraically closed field $k$, a bijection of $A_g(k)$ with the set of isomorphism classes of principally polarized abelian varieties over $k$.

Since any principally polarized complex torus $(X,E)$ of dimension $g$ is algebraizable, $(X,E)$ defines an isomorphism class in the category $\mathcal{A}_g(\CC)$ that we shall denote $[(X,E)]$. Let  
\begin{align*}
j_g : \mathbf{H}_g &\to A_g(\CC)\\
               \tau &\mapsto [(\mathbf{X}_{g,\tau},E_{g,\tau})]\text{.}
\end{align*}

The next result follows immediately from our proof of the Lemma \ref{fielddef}.

\begin{lemma}\label{remarkfielddef}
For any $\tau \in \mathbf{H}_g$, the smallest algebraically closed subfield over which $\mathbf{X}_{g,\tau}$ is definable is given by the algebraic closure in $\CC$ of the residue field $\mathbf{Q}(j_g(\tau))$. 
\end{lemma}

\begin{proof}[Proof of Theorem \ref{trdeg}]
Let $\tau \in \mathbf{H}_g$ and fix any integer $n\ge 3$. The principally polarized complex torus $(\mathbf{X}_{g,\tau},E_{g,\tau})$ endowed with the integral symplectic basis modulo $n$ induced by $\beta_{g,\tau}$ (cf. Remark \ref{localift}) defines a complex point $u(\tau) \in A_{g,n}(\CC)$ in the fine moduli space $A_{g,n}$ over $\ZZ[1/n,\zeta_n]$ of principally polarized abelian varieties of dimension $g$ endowed with a symplectic basis of its $n$-torsion subscheme (cf. \ref{level}).

 Let $(X_{g,n},\lambda_{g,n})$ denote the universal principally polarized abelian scheme over $A_{g,n}$. Then, by the remark following \cite{fonseca16} Definition 2.4, there exists a Zariski open neighborhood $U\subset A_{g,n,\overline{\QQ}}$ of $u(\tau)$ over which $(X_{g,n,\overline{\QQ}},\lambda_{g,n})$ admits a symplectic-Hodge basis $b$. Let us denote by $(X,\lambda)=(X_{g,n,\overline{\QQ}},\lambda_{g,n})\times_{A_{g,n,\overline{\mathbf{Q}}}} U$ the restriction of $(X_{g,n,\overline{\QQ}},\lambda_{g,n})$ to $U$.

In the following, fiber products will be taken with respect to $\overline{\QQ}$. The symplectic-Hodge basis $b$ of $(X,\lambda)_{/U}$ induces an isomorphism of principal $P_g$-bundles over $U$
\begin{align*}
P_{g,\overline{\QQ}}\times U &\stackrel{\sim}{\to} B(X,\lambda) \\
        (p,u) &\mapsto b_u\cdot p\text{,} 
\end{align*} 
where $B(X,\lambda)$ is the $U$-scheme defined in \cite{fonseca16} Corollary 3.4. By composing this isomorphism with the isomorphism in Lemma \ref{isomorphismgsp}, we obtain the isomorphism of $\overline{\QQ}$-schemes
\begin{align*}
f: {\GSp}_{2g,\overline{\QQ}}^*\times U &\to \mathbf{G}_{m,\overline{\QQ}} \times  {\Sym}_{g,\overline{\QQ}} \times B(X,\lambda)\\
               (s,u) &\mapsto (\nu(s), \tau(s), b_u\cdot p(s))\text{.}
\end{align*}

Note that the canonical morphism $h: B(X,\lambda) \to \mathcal{B}_{g,\overline{\QQ}}\cong B_{g,\overline{\QQ}}$ is quasi-finite, since it fits into the cartesian diagram of Deligne-Mumford stacks
$$
  \raisebox{-0.5\height}{\includegraphics{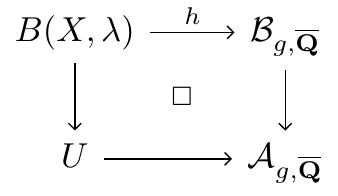}}
$$
where the bottom arrow $U \to \mathcal{A}_{g,\overline{\QQ}}$ is given by the composition of the open immersion $U\subset A_{g,n,\overline{\QQ}}$ with the (canonical) finite étale morphism $A_{g,n,\overline{\QQ}} \to \mathcal{A}_{g,\overline{\QQ}}$.

In particular, by composing $f$ with $h$, we obtain a quasi-finite morphism of $\overline{\QQ}$-schemes
\begin{align*}
q: {\GSp}_{2g,\overline{\QQ}}^*\times U \to \mathbf{G}_{m,\overline{\QQ}} \times {\Sym}_{g,\overline{\QQ}} \times B_{g,\overline{\QQ}}
\end{align*}
given on geometric points by
\begin{align*}
q(s,u) = (\nu(s),\tau(s),[(X_u,\lambda_u,b_u\cdot p(s))])
\end{align*}
where $[(X_u,\lambda_u,b_u\cdot p(s))]$ denotes the isomorphism class in $\mathcal{B}_g(k(u))$ of $(X_u,\lambda_u,b_u\cdot p(s))$, and $k(u)$ denotes the residue field of $u\in U$.

Let $F: (\mathbf{X}_{g,\tau},E_{g,\tau}) \stackrel{\sim}{\to} (X^{\an}_{u(\tau)},E_{\lambda_{u(\tau)}})$ be the isomorphism of principally polarized complex tori corresponding, by the universal property of $A_{g,n}$, to the reduction of the integral symplectic basis $\beta_{g,\tau}$ modulo $n$, and put
\begin{align*}
s(\tau) \defeq \frac{1}{2\pi i}P(X_{u(\tau)}^{\an},E_{\lambda_{u(\tau)}}, b_{u(\tau)},F_*\beta_{g,\tau})  \in {\GSp}_{2g}^*(\CC)\text{.}
\end{align*}
It follows from Corollary \ref{caraceta} and Lemma \ref{computationallemma} that $F^*b_{u(\tau)}\cdot p(s(\tau)) = \bfb_{g,\tau}$, so that
\begin{align*}
[(X_{u(\tau)}, \lambda_{u(\tau)},b_{u(\tau)}\cdot p(s(\tau)))]= [(\mathbf{X}_{g,\tau},E_{g,\tau}, \bfb_{g,\tau})] = \varphi_g(\tau)\text{.}
\end{align*}
Thus, by Lemma \ref{lemmaperiodmatrix}, we obtain
\begin{align*}
q(s(\tau),u(\tau)) = \left(\frac{1}{2\pi i}, \tau, \varphi_g(\tau)\right) \in \mathbf{G}_m(\CC)\times {\Sym}_g(\CC) \times B_g(\CC)\text{.}
\end{align*}


To finish the proof, it is sufficient to show that $\mathcal{P}(\mathbf{X}_{g,\tau})$ is an algebraic extension of $\overline{\QQ}(s(\tau),u(\tau))$. For this, let
\begin{align*}
h: A_{g,n,\overline{\QQ}} \to A_{g,\overline{\QQ}}
\end{align*}
be the canonical map; then $h(u(\tau))=j_g(\tau)$. As $h$ is finite (in fact, it identifies $A_{g,\overline{\QQ}}$ with the quotient $\Sp_{2g}(\ZZ/n\ZZ)\setminus A_{g,n,\overline{\QQ}}$), the residue field $\overline{\QQ}(u(\tau))$ is a finite extension of $\overline{\QQ}(j_g(\tau))$. Then, it follows from Corollary \ref{remarkfielddef} that $\overline{\QQ}(u(\tau))$ is contained in the smallest algebraically closed field of $\CC$ over which $\mathbf{X}_{g,\tau}$ is definable (namely, the algebraic closure in $\CC$ of $\overline{\QQ}(j_g(\tau))$). Finally, since $2\pi i \in \mathcal{P}(\mathbf{X}_{g,\tau})$ by Lemma \ref{lemmaperiodmatrix} (1), the assertion follows from Remark \ref{remarkperiodmatrix}.
\end{proof}

\section{Group-theoretic interpretation of $B_g(\CC)$ and of the higher Ramanujan vector fields}\label{gpinterpret}

In this section we shall explain how to realize the complex manifold $B_g(\CC)$ as a domain (in the analytic topology) of the quotient manifold $\Sp_{2g}(\ZZ)\backslash\Sp_{2g}(\CC)$ (Corollary \ref{realization}).

We shall also give an explicit expression for the higher Ramanujan vector fields, and for the holomorphic map $\varphi_g: \mathbf{H}_g \to B_g(\CC)$, under this group-theoretic interpretation. For this, recall that the Lie algebra of $\Sp_{2g}(\CC)$ is given by
\begin{align*}
\Lie {\Sp}_{2g} (\CC) = \left\{\left.\left(\begin{array}{cc}A & B \\ C & D \end{array}\right) \in M_{2g\times 2g}(\CC)\right|B\transp=B\text{, }C\transp = C\text{, }D = -A\transp\right\}\text{.}
\end{align*}
For $1\le k \le l \le g$, let us consider the left invariant holomorphic vector field $\tilde{V}_{kl}$ on $\Sp_{2g}(\CC)$ corresponding to
\begin{align*}
\frac{1}{2\pi i}\left(\begin{array}{cc}0 & \mathbf{E}^{kl} \\ 0 & 0 \end{array}\right) \in \Lie {\Sp}_{2g}(\CC)\text{;}
\end{align*}
it descends to a holomorphic vector field $V_{kl}$ on the quotient $\Sp_{2g}(\ZZ)\backslash \Sp_{2g}(\CC)$.

\begin{theorem}\label{unifhrvf}
  Let $(v_{kl})_{1\le k \le l \le g}$ be the higher Ramanujan vector fields on $B_g(\CC)$. Under the identification of $B_g(\CC)$ with an open submanifold of $\Sp_{2g}(\ZZ)\backslash \Sp_{2g}(\CC)$ of Corollary \ref{realization}, we have:
  \begin{enumerate}
     \item For every $1\le k \le l \le g$,
\begin{align*}
v_{kl} = V_{kl}|_{B_g(\CC)}\text{.}
\end{align*}
\item The solution of the higher Ramanujan equations $\varphi_g: \mathbf{H}_g \to B_g(\CC)$ is given by
  \begin{align*}
    \varphi_g(\tau) = {\Sp}_{2g}(\ZZ)\left(\begin{array}{cc}\mathbf{1}_g & \tau \\ 0 &\mathbf{1}_g\end{array} \right) \in {\Sp}_{2g}(\ZZ)\backslash {\Sp}_{2g}(\CC)\text{.}
    \end{align*}
   \end{enumerate}
 \end{theorem}

 As an example of application, we shall prove the following easy consequence of the above theorem.
 
 \begin{coro}\label{ramleafclosed}
The image of $\varphi_g: \mathbf{H}_g \to B_g(\CC)$ is closed for the analytic topology.
\end{coro}

\subsection{Realization of $B_g(\CC)$ as an open submanifold of $\Sp_{2g}(\ZZ)\backslash\Sp_{2g}(\CC)$} \label{parareal}

Let $\mathbf{B}_g = B(\mathbf{X}_g,E_g)$ be the principal $P_g(\CC)$-bundle over $\mathbf{H}_g$ associated to the principally polarized complex torus $(\mathbf{X}_g,E_g)_{/\mathbf{H}_g}$ as defined in Lemma \ref{relrepr}, so that the fiber of $\mathbf{B}_g \to \mathbf{H}_g$ over $\tau \in \mathbf{H}_g$ is given by the set of symplectic-Hodge bases of $(\mathbf{X}_{g,\tau},E_{g,\tau})$. 

We shall first realize $\mathbf{B}_g$ as a ``period domain'' in $\Sp_{2g}(\CC)$. For this, let us introduce the following convenient modification of period matrices (Definition \ref{defiperiodmatrix}).

\begin{defi}
  Let $(X,E)$ be a principally polarized complex torus of dimension $g$, and $b$ (resp. $\beta$) be a symplectic-Hodge basis (resp. an integral symplectic basis) of $(X,E)$. Let
  \begin{align*}
P(X,E,b,\beta) = \left(\begin{array}{cc}
                               \Omega_1 & N_1 \\
                               \Omega_2 & N_2
                               \end{array}\right) \in {\GSp}_{2g}(\CC)
\end{align*}
be the period matrix of $(X,E)$ with respect to $b$ and $\beta$. We define
\begin{align*}
\Pi(X,E,b,\beta) \defeq \left(\begin{array}{cc}
                               N_2 & \frac{1}{2\pi i}\Omega_2 \\[0.5em]
                               N_1 & \frac{1}{2\pi i}\Omega_1
                               \end{array}\right) \in {\Sp}_{2g}(\CC)
\end{align*}
Observe that this matrix is indeed symplectic by Lemma \ref{lemmaperiodmatrix}.
\end{defi}

We define a holomorphic map $\Pi : \mathbf{B}_g \to {\Sp}_{2g}(\CC)$ as follows. Let $q$ be a point in $\mathbf{B}_g$ lying above $\tau \in \mathbf{H}_g$, and corresponding to the symplectic-Hodge basis $b$ of $(\mathbf{X}_{g,\tau},E_{g,\tau})$, then
\begin{align*}
\Pi(q) \defeq \Pi (\mathbf{X}_{g,\tau},E_{g,\tau},b,\beta_{g,\tau})
\end{align*}
where $\beta_g$ is the integral symplectic basis of $(\mathbf{X}_g,E_{g})_{/\mathbf{H}_g}$ defined in Example \ref{intsymplbasis}. 

\begin{obs} \label{moduliinterpret}
Alternatively, recall that $\mathbf{H}_g$ may be regarded as the moduli space for principally polarized complex tori of dimension $g$ endowed with an integral symplectic basis (Proposition \ref{reprsintsympl}). In particular, as already remarked in the proof of Proposition \ref{reprsympl}, points in $\mathbf{B}_g$ correspond to isomorphism classes $[(X,E,b,\beta)]$ of quadruples $(X,E,b,\beta)$, where $(X,E)$ is a principally polarized complex torus of dimension $g$, and $b$ (resp. $\beta$) is a symplectic-Hodge basis (resp. integral symplectic basis) of $(X,E)$. Under this identification, the map $\Pi : \mathbf{B}_g \to \Sp_{2g}(\CC)$ is given by $[(X,E,b,\beta)] \mapsto \Pi (X,E,b,\beta)$.
\end{obs}
 
Let us consider the moduli-theoretic interpretation of $\mathbf{B}_g$ of the above remark, and recall that $\mathbf{B}_g$ is endowed with a natural left action of the discrete group $\Sp_{2g}(\ZZ)$ given by
\begin{align*}
\left(\begin{array}{cc}A & B \\ C & D \end{array}\right) \cdot [(X,E,b,\beta)] = \left[\left(X,E,b, \beta\cdot \left(\begin{array}{cc}
D\transp & B\transp \\
C\transp & A\transp
\end{array}\right)
\right)\right]
\end{align*} (cf. Remark \ref{actionsp}), and a right action of the Siegel parabolic subgroup $P_g(\CC)\le \Sp_{2g}(\CC)$ given by
\begin{align*}
[(X,E,b,\beta)]\cdot p = [(X,E,b\cdot p,\beta)]\text{,}
\end{align*}
where both $\beta$ and $b$ are regarded as row vectors of order $2g$.

 Let us denote by $P'_g$ the subgroup scheme of $\Sp_{2g}$ consisting of matrices $(A \ B \ ; C \ D )$ such that $B=0$. A simple computation proves the following equivariance properties of $\Pi : \mathbf{B}_g \to \Sp_{2g}(\CC)$.

\begin{lemma}\label{equivariancepi}
Consider the isomorphism of groups
\begin{align*}
P_g(\CC) &\stackrel{\sim}{\to} P'_g(\CC) \\
 p=\left( \begin{array}{cc}
    A & B \\
    0 & (A\transp)^{-1}
    \end{array}\right) &\mapsto p' \defeq \left( \begin{array}{cc}
    (A\transp)^{-1} & 0 \\
    2\pi i B & A
    \end{array}\right)\text{.}
\end{align*} 
Then, for any $q \in \mathbf{B}_g$, $\gamma \in \Sp_{2g}(\ZZ)$, and $p \in P_g(\CC)$, we have
\begin{align*}
\Pi(\gamma\cdot q) = \gamma\Pi(q)\ \ \ \text{ and }\ \ \ \Pi(q\cdot p) = \Pi(q)p'
\end{align*}
in $\Sp_{2g}(\CC)$.
\end{lemma}

Let us now consider the \emph{Lagrangian Grassmannian}, namely the smooth and quasi-projective $\CC$-scheme of dimension $g(g+1)/2$ obtained as the quotient of complex affine algebraic groups
\begin{align*}
L_g \defeq {\Sp}_{2g,\CC}/P'_{g,\CC}\text{.}
\end{align*}
The complex manifold $L_g(\CC)=\Sp_{2g}(\CC)/P'_g(\CC)$ may be naturally identified with the quotient of
\begin{align*}
  M\defeq \{(Z_1,Z_2) \in M_{g\times g}(\CC)\times M_{g\times g}(\CC) \mid Z_1\transp Z_2 = Z_2\transp Z_1\text{, } \text{rank} (Z_1 \ Z_2) = g\}
\end{align*}
by the right action of $\GL_g(\CC)$ defined by matrix multiplication:
\begin{align*}
(Z_1,Z_2) \cdot S \defeq (Z_1S,Z_2S)\text{.}
\end{align*}
We denote the class in $L_g(\CC)$ of a point $(Z_1,Z_2)\in M$ by $(Z_1:Z_2)$. The canonical map
\begin{align*}
\pi: {\Sp}_{2g,\CC} \to L_g
\end{align*}
is then given on complex points by
\begin{align*}
\pi \left(\begin{array}{cc}A & B \\ C & D \end{array} \right) = (B:D)\text{.}
\end{align*}

\begin{prop}\label{prop1}
  Let $\iota: \mathbf{H}_g \to L_g(\CC)$ be the open embedding given by $\iota(\tau)=(\tau: \mathbf{1}_g)$. Then the diagram of complex manifolds
$$
  \raisebox{-0.5\height}{\includegraphics{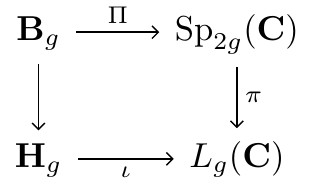}}
$$ 
is cartesian. That is, $\Pi: \mathbf{B}_g \to \Sp_{2g}(\CC)$ induces a biholomorphism of $\mathbf{B}_g$ onto the open submanifold
\begin{align*}
\pi^{-1}(\iota(\mathbf{H}_g))=\left\{\left.\left(\begin{array}{cc}A & B \\ C & D \end{array} \right) \in {\Sp}_{2g}(\CC) \right| D \in {\GL}_{g}(\CC)\text{, }BD^{-1} \in \mathbf{H}_g \right\}
\end{align*}
of $\Sp_{2g}(\CC)$, and makes the above diagram commute.
\end{prop}

\begin{proof}
The commutativity of the diagram in the statement is easy (cf. proof of Proposition \ref{reprsintsympl}). In particular, if $q,q'\in \mathbf{B}_g$ satisfy $\Pi(q)=\Pi(q')$, then they lie above the same point $\tau \in \mathbf{H}_g$. Let $b$ (resp. $b'$) be the symplectic-Hodge basis of $(\mathbf{X}_{g,\tau},E_{g,\tau})$ corresponding to $q$ (resp. $q'$). Since period matrices are base change matrices for the comparison isomorphism,  and
\begin{align*}
\Pi(\mathbf{X}_{g,\tau}, E_{g,\tau},b, \beta_{g,\tau})=\Pi(\mathbf{X}_{g,\tau}, E_{g,\tau},b', \beta_{g,\tau})\text{,}
\end{align*}
it is clear that $b=b'$. This proves that $\Pi$ is injective.

Observe that $\mathbf{B}_g$ and $\Sp_{2g}(\CC)$ are complex manifolds of same dimension. Thus, to finish our proof, it suffices to check that  $\Pi(\mathbf{B}_g)=\pi^{-1}(\iota(\mathbf{H}_g))$ (\cite{GH78} p. 19). Let $s \in\pi^{-1}(\iota(\mathbf{H}_g))$, and let $\tau \in \mathbf{H}_g$ be such that $\iota(\tau)=\pi(s)$. Fix any $q \in \mathbf{B}_g$ lying above $\tau \in \mathbf{H}_g$. Then, there exists a unique $p' \in P_g'(\CC)$ such that $s=\Pi(q)p'$. Hence, by Lemma \ref{equivariancepi}, $s= \Pi(q\cdot p) \in \Pi(\mathbf{B}_g)$.
\end{proof}

Recall from Proposition \ref{reprsympl} that the canonical map 
\begin{align}\label{canmap}
\begin{split}
\mathbf{B}_g &\to B_g(\CC)\\
       [(X,E,b,\beta)] &\mapsto [(X,E,b)]
\end{split}
\end{align}
induces a biholomorphism
\begin{align*}
{\Sp}_{2g}(\ZZ) \backslash \mathbf{B}_g \stackrel{\sim}{\to} B_g(\CC)\text{.}
\end{align*}
Furthermore, note that Lemma \ref{equivariancepi} implies that the action of $\Sp_{2g}(\ZZ)$ on $\Sp_{2g}(\CC)$ by left multiplication preserves the open subset $\Pi(\mathbf{B}_g)$.

\begin{coro} \label{realization}
The map $\Pi : \mathbf{B}_g \to \Sp_{2g}(\CC)$ induces a biholomorphism of $B_g(\CC)$ onto the open submanifold of $\Sp_{2g}(\ZZ) \backslash \Sp_{2g}(\CC)$
\begin{align*}
{\Sp}_{2g}(\ZZ)\setminus \Pi(\mathbf{B}_g) = \{{\Sp}_{2g}(\ZZ) s  \in {\Sp}_{2g}(\ZZ) \backslash {\Sp}_{2g}(\CC) \mid \pi(s) \in \iota(\mathbf{H}_g) \}\text{.}
\end{align*}
\end{coro}

\subsection{Proof of Theorem \ref{unifhrvf} and of Corollary \ref{ramleafclosed}}





We prove parts (1) and (2) of Theorem \ref{unifhrvf} separately.

\begin{proof}[Proof of Theorem \ref{unifhrvf} (1)]
  It is sufficient to prove that the solutions of the differential equations defined by $v_{kl}$ and by $V_{kl}$ coincide. More precisely, let $U$ be a simply connected open subset of $\mathbf{H}_g$, and $u:U \to B_g(\CC)$ be a solution of the higher Ramanujan equations (Definition \ref{defihreq}); we shall prove that, for any lifting
 $$
  \raisebox{-0.5\height}{\includegraphics{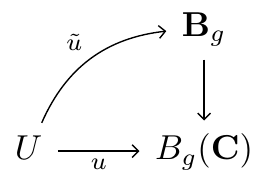}}
$$
 of $u$, the holomorphic map $h \defeq \Pi \circ \tilde{u} : U \to \Sp_{2g}(\CC)$ is a solution of the differential equations
\begin{align} \label{rameqdisg}
\theta_{kl}h = \tilde{V}_{kl}\circ h  \text{, }\ \ \ 1\le k \le l \le g \text{.}
\end{align}
where $\theta_{kl} = \frac{1}{2\pi i}\frac{\partial}{\partial \tau_{kl}}$.

By the universal property of $\mathbf{B}_g$, the holomorphic map $\tilde{u}$ corresponds to a principally polarized complex torus $(X,E)$ over $U$, of relative dimension $g$, endowed with a symplectic-Hodge basis $b = (\omega_1,\ldots,\omega_g,\eta_1,\ldots,\eta_g)$ and an integral symplectic basis $\beta = (\gamma_1,\ldots,\gamma_g,\delta_1,\ldots,\delta_g)$. For $\tau\in U$, let us write
\begin{align*}
h(\tau) = \left(\begin{array}{cc} N_2(\tau) & \frac{1}{2\pi i}\Omega_2(\tau)\\[0.5em]
                               N_1(\tau) & \frac{1}{2\pi i}\Omega_1(\tau) \end{array} \right) \in {\Sp}_{2g}(\CC)
\end{align*}
where $\Omega_1,\Omega_2,N_1,N_2: U \to M_{g\times g}(\CC)$ are holomorphic.

Now, since $u$ is a solution of the higher Ramanujan equations, it follows from Proposition \ref{equivalences} (3) that, for every $1\le i \le j \le g$,
\begin{enumerate}
   \item[(i)] $\theta_{ij}\Omega_1 = N_1\mathbf{E}^{ij}$, $\theta_{ij}\Omega_2 = N_2\mathbf{E}^{ij}$
   \item[(ii)] $\theta_{ij}N_1 = 0$, $\theta_{ij}N_2=0$.
\end{enumerate}
As $U$ is connected, (ii) implies that $N_1$ and $N_2$ are constant. Thus, (i) implies that $\frac{1}{2\pi i}\Omega_1 -  N_1\tau$ and $\frac{1}{2\pi i}\Omega_2 -  N_2\tau$ are also constant. In other words, there exists a unique element $s \in \Sp_{2g}(\CC)$ such that
\begin{align*}
h(\tau) = s\left(\begin{array}{cc}\mathbf{1}_g & \tau \\ 0 &\mathbf{1}_g\end{array} \right)
\end{align*}
for every $\tau \in U$. Finally, since each $\tilde{V}_{kl}$ is left invariant, it is easy to see that $h$ is a solution of the differential equations (\ref{rameqdisg}).
\end{proof}



\begin{lemma} \label{psi}
For any $\tau \in \mathbf{H}_g$, we have
  \begin{align*}
    \Pi (\mathbf{X}_{g,\tau},E_{g,\tau}, \bfb_{g,\tau},\beta_{g,\tau}) = \left(\begin{array}{cc} \mathbf{1}_g & \tau \\ 0 & \mathbf{1}_g\end{array} \right)\text{.}
  \end{align*}
\end{lemma}

\begin{proof}
 Let us write
  $$
  \Pi (\mathbf{X}_{g,\tau},E_{g,\tau}, \bfb_{g,\tau},\beta_{g,\tau}) = \left(\begin{array}{cc}
                               N_2(\tau) & \frac{1}{2\pi i}\Omega_2(\tau)  \\[0.5em]
                               N_1(\tau) & \frac{1}{2\pi i}\Omega_1(\tau)
                               \end{array}\right)\text{.}
                             $$
By definition of $\beta_g$ and of $\bfb_g$, it is clear that $\Omega_1(\tau) = 2\pi i\mathbf{1}_g$ and that $\Omega_{2}(\tau) = 2\pi i \tau$. That $N_1(\tau) = 0$ and $N_2(\tau)=\mathbf{1}_g$ is a reformulation of Corollary \ref{caraceta}.
\end{proof}

\begin{proof}[Proof of Theorem \ref{unifhrvf} (2)]
By definition, $\varphi_g$ is given by the composition of
\begin{align*}
\mathbf{H}_g &\to \mathbf{B}_g\\
         \tau &\mapsto [(\mathbf{X}_{g,\tau},E_{g,\tau}, \bfb_{g,\tau},\beta_{g,\tau})]
\end{align*}
with the canonical map $\mathbf{B}_g \to B_g(\CC)$. The result now follows from Lemma \ref{psi}.
\end{proof}

\begin{proof}[Proof of Corollary \ref{ramleafclosed}]
  Consider the subgroup
  $$
  U_g(\CC) \defeq \left.\left\{\left(\begin{array}{cc} \mathbf{1}_g & Z \\ 0 & \mathbf{1}_g\end{array}\right) \in M_{2g\times 2g}(\CC) \right| Z\transp =Z\right\} \le {\Sp}_{2g}(\CC)\text{.}
  $$
  The statement is equivalent to asserting that the image of $U_g(\CC)\subset \Sp_{2g}(\CC)$ in the quotient $\Sp_{2g}(\ZZ)\backslash \Sp_{2g}(\CC)$ is closed, or, equivalently, that $\Sp_{2g}(\ZZ)\cdot U_g(\CC) \subset \Sp_{2g}(\CC)$ is closed. Let us consider the (holomorphic) map
  \begin{align*}
    f: {\Sp}_{2g}(\CC) &\to M_g(\CC)\times M_g(\CC)\\
                \left(\begin{array}{cc}A & B\\ C & D \end{array} \right)&\mapsto (A,C)\text{.}
  \end{align*}
  Now, one simply remarks that
\begin{align*}
{\Sp}_{2g}(\ZZ)\cdot U_g(\CC) = f^{-1}(f({\Sp}_{2g}(\ZZ)))\text{.}
\end{align*}
Since $f(\Sp_{2g}(\ZZ))\subset M_g(\ZZ) \times M_g(\ZZ)$, and $M_g(\ZZ)\times M_g(\ZZ)$ is a closed discrete subset of $M_g(\CC)\times M_g(\CC)$ for the analytic topology, we conclude that  ${\Sp}_{2g}(\ZZ)\cdot U_g(\CC)$ is closed in $\Sp_{2g}(\CC)$.
\end{proof}

\section{Zariski-density of leaves of the higher Ramanujan foliation}


 Let us denote by $\mathcal{R}_g$ the subbundle of the holomorphic tangent bundle $T_{B_g(\CC)}$ generated by the higher Ramanujan vector fields $v_{ij}$, $1\le i \le j \le g$. Since the vector fields $v_{ij}$ commute (\cite{fonseca16} Corollary 5.10), $\mathcal{R}_g$ is an integrable subbundle of $T_{B_g(\CC)}$. Hence, by holomorphic Frobenius Theorem, $\mathcal{R}_g$ induces a holomorphic foliation on $B_g(\CC)$; we call it the \emph{higher Ramanujan foliation}.


 Using the group-theoretic interpretation of $B_g(\CC)$ of Section \ref{gpinterpret}, we shall also provide an explicit parametrization of every leaf $L\subset B_g(\CC)$\footnote{By definition, a \emph{leaf} of the higher Ramanujan foliation on $B_g(\CC)$ is a maximal connected immersed complex submanifold of $B_g(\CC)$ that is everywhere tangent to $\mathcal{R}_g$.} of the higher Ramanujan foliation (see Proposition \ref{charactleaves} for a precise statement).

Our main results in this section are the following Zariski-density statements.

\begin{theorem} \label{densite}
Every leaf $L\subset B_g(\CC)$ of the higher Ramanujan foliation is Zariski-dense in $B_{g,\CC}$, that is, for every closed subscheme $Y$ of $B_{g,\CC}$, if $Y(\CC)$ contains $L$, then $Y(\CC)=B_g(\CC)$. 
\end{theorem}

In particular, we obtain that the image of the solution of the higher Ramanujan equations $\varphi_g: \mathbf{H}_g \to B_g(\CC)$ is Zariski-dense in $B_{g}$.

Concerning the image of $\varphi_g$, we can actually derive the following \emph{a priori} stronger result. 

\begin{coro} \label{graphe}
The set $ \{(\tau,\varphi_g(\tau)) \in {\Sym}_g(\CC)\times B_g(\CC) \mid \tau \in \mathbf{H}_g\}$ is  Zariski-dense in $\Sym_{g,\CC}\times_{\CC} B_{g,\CC}$.
\end{coro}


The proof of both Zariski-density results will rely on the following elementary lemma.

\begin{lemma}[Fibration method] \label{metfib}
Let $p:X \to S$ be a morphism of separated $\CC$-schemes of finite type and let $E\subset X(\CC)$ be a subset. If, for every $s \in p(E)$, the set $E\cap X_s$ is Zariski-dense in $X_s\defeq p^{-1}(s)$, and one of the following conditions is satisfied, 
\begin{enumerate}[(i)]
    \item $p(E) = S(\CC)$,
    \item $p$ is open (in the Zariski topology) and $p(E)$ is Zariski-dense in $S$,
\end{enumerate}
then $E$ is Zariski-dense in $X$.
\end{lemma}

\begin{proof}
Let $U$ be a non-empty Zariski open subset of $X$; we must show that $E\cap U$ is non-empty. In both cases (i) and (ii) above, there exists a closed point $s \in p(E)\cap p(U)$. Since $E\cap X_s$ is Zariski-dense in $ X_s$ and $U\cap X_s$ is a non-empty open subset of $X_s$, there exists a closed point $x\in E\cap U\cap X_s\subset E\cap U$. 
\end{proof}

\subsection{Characterization of the leaves of the higher Ramanujan foliation}

\subsubsection{}\label{unipotentfoliation} Let $U_g$ be the unipotent subgroup scheme of $\Sp_{2g}$ defined by
\begin{align*}
U_g(R) = \left\{\left. \left(\begin{array}{cc} \mathbf{1}_g & Z \\ 0 & \mathbf{1}_g \end{array}\right) \in M_{2g\times 2g}(R)\right| Z\transp = Z \right\}
\end{align*}
for any ring $R$.

The Lie algebra of $U_g(\CC)$ is given by
\begin{align*}
\Lie U_g(\CC) = \left\{\left. \left(\begin{array}{cc} 0 & Z \\ 0 & 0 \end{array}\right) \in M_{2g\times 2g}(\CC)\right| Z\transp = Z \right\}\text{,}
\end{align*}
and admit as a basis the vectors
\begin{align*}
\frac{1}{2\pi i}\left(\begin{array}{cc}0 & \mathbf{E}^{kl} \\ 0 & 0 \end{array}\right) \in \Lie {U}_{g}(\CC)\text{, } \ \ \ 1\le k \le l \le g\text{,}
\end{align*}
inducing the higher Ramanujan vector fields on the quotient $\Sp_{2g}(\ZZ)\backslash \Sp_{2g}(\CC)$ (Section \ref{gpinterpret}). In particular, under the realization of $B_g(\CC)$ as an open submanifold of $\Sp_{2g}(\ZZ)\backslash \Sp_{2g}(\CC)$ of Corollary \ref{realization}, the higher Ramanujan foliation on $B_g(\CC)$ is induced by the foliation on $\Sp_{2g}(\CC)$ defined by $U_g(\CC)$, i.e. the foliation whose leaves are left cosets of $U_g(\CC)$ in $\Sp_{2g}(\CC)$.

It follows from the above discussion that, under the identification of $\mathbf{B}_g$ (resp. $B_g(\CC)$) with an open submanifold of $\Sp_{2g}(\CC)$ (resp. $\Sp_{2g}(\ZZ)\backslash \Sp_{2g}(\CC)$) via $\Pi$ (cf. Proposition \ref{prop1} and Corollary \ref{realization}), for any leaf $L$ of the higher Ramanujan foliation on $B_g(\CC)$, there exists $\delta \in \Sp_{2g}(\CC)$ such that $L$ is a connected component of the image of $\delta U_g(\CC)\cap \mathbf{B}_g$ in $B_g(\CC)$ under the quotient map $\Sp_{2g}(\CC)\to \Sp_{2g}(\ZZ)\backslash \Sp_{2g}(\CC)$. We shall provide a more precise result in Proposition \ref{charactleaves}.

\subsubsection{} We may also obtain an explicit \emph{parametrization} of every leaf. For this, let us consider $\Sym_g(\CC) = \{Z \in M_{g\times g}(\CC) \mid Z\transp = Z\}$ as an open subset of the Lagrangian Grassmannian $L_g(\CC)$ (cf. discussion preceding Proposition \ref{prop1}) via
\begin{align*}
  {\Sym}_g(\CC) &\to L_g(\CC)\\
          Z &\mapsto (Z:\mathbf{1}_g)\text{,}
\end{align*}
so that the embedding $\iota : \mathbf{H}_g \to L_g(\CC)$ defined in Proposition \ref{prop1} is given by the restriction of $\Sym_g(\CC) \to L_g(\CC)$ to $\mathbf{H}_g$. Furthermore, let
\begin{align*}
  \psi : {\Sym}_g(\CC) &\to {\Sp}_{2g}(\CC)\\
                 Z &\mapsto \left(\begin{array}{cc}\mathbf{1}_g & Z \\ 0 & \mathbf{1}_g \end{array}\right)\text{.}
 \end{align*}

\begin{obs}
Under the obvious identification of $\Sym_g(\CC)$ with $\Lie U_g(\CC)$, the map $\psi$ is simply the exponential $\exp : \Lie U_g(\CC) \to U_g(\CC)\subset \Sp_{2g}(\CC)$.
\end{obs}

Now, the action of $\Sp_{2g}(\CC)$ on itself by left multiplication descends to a left action of $\Sp_{2g}(\CC)$ on $L_g(\CC)$ given explicitly by
\begin{align*}
  \left( \begin{array}{cc}
           A & B \\
           C & D
 \end{array}\right)\cdot (Z_1:Z_2) = (AZ_1+BZ_2:CZ_1+DZ_2)\text{.}
\end{align*}
For any $\delta \in \Sp_{2g}(\CC)$, let us define
\begin{align*}
  \psi_{\delta} : \delta^{-1}\cdot {\Sym}_g(\CC)\subset L_g(\CC) &\to {\Sp}_{2g}(\CC)\\
                          p &\mapsto \delta^{-1}\psi(\delta\cdot p)\text{.}
\end{align*}
Then $\psi_{\delta}$ induces a biholomorphism of $\delta^{-1}\cdot {\Sym}_g(\CC)$ onto the closed submanifold $\delta^{-1}U_g(\CC)\subset \Sp_{2g}(\CC)$.

We put
\begin{align*}
U_{\delta} \defeq \{\tau \in \mathbf{H}_g \mid \delta\cdot (\tau:1) \in {\Sym}_g(\CC)\subset L_g(\CC)\} = (\delta^{-1}\cdot {\Sym}_g(\CC))\cap \mathbf{H}_g  \text{.}
\end{align*}
Equivalently, if $\delta = (A \ B \ ; \ C \ D)$, then
\begin{align*}
U_{\delta} = \{\tau \in \mathbf{H}_g \mid C\tau +D \in {\GL}_g(\CC)\}\text{.}
\end{align*}

\begin{defi}
  For any $\delta \in \Sp_{2g}(\CC)$, we define a holomorphic map $\varphi_{\delta} : U_{\delta} \to B_g(\CC) \subset \Sp_{2g}(\ZZ)\backslash \Sp_{2g}(\CC)$ by
 \begin{align*}
\varphi_{\delta}(\tau) \defeq {\Sp}_{2g}(\ZZ)\psi_{\delta}(\tau)
\end{align*}
for any $\tau \in U_{\delta}$.
\end{defi}

Note that $\psi_{\delta}(U_{\delta}) = \delta^{-1}U_g(\CC) \cap \mathbf{B}_g \subset \Sp_{2g}(\CC)$ by Lemma \ref{equivariancepi}. In particular, the image of $\varphi_{\delta}$ is indeed in $B_g(\CC)$ . Moreover, if $\delta \in U_g(\CC)$, then $U_{\delta} = \mathbf{H}_g$ and $\varphi_{\delta}=\varphi_g$ (cf. Theorem \ref{unifhrvf} (2)).

\begin{lemma}\label{Uconnected}
For any $\delta \in \Sp_{2g}(\CC)$, $U_{\delta}$ is a dense connected open subset of $\mathbf{H}_g$.
\end{lemma}

\begin{proof}
Let $\delta = (A \ B \ ; \ C \ D) \in \Sp_{2g}(\CC)$. By definition, $U_{\delta}$ is the complement in $\mathbf{H}_g$ of the codimension 1 analytic subset $\{\tau \in \mathbf{H}_g \mid \det(C\tau+D)=0\}$. It is thus a dense open subset of $\mathbf{H}_g$. Since $\mathbf{H}_g$ is a connected open subset of an affine space, it follows from Riemann's extension theorem (cf. \cite{huybrechts05} Proposition 1.1.7)  that $U_{\delta}$ is connected. 
\end{proof}

\begin{prop}\label{charactleaves}
For every $\delta \in \Sp_{2g}(\CC)$, the image of the map $\varphi_{\delta} : U_{\delta} \to B_g(\CC)$ is a leaf of the higher Ramanujan foliation on $B_g(\CC)$, and  coincides with the image of $\delta^{-1}U_g(\CC)\cap \mathbf{B}_g$ in  $B_g(\CC)$ under the quotient map $\Sp_{2g}(\CC) \to \Sp_{2g}(\ZZ)\backslash \Sp_{2g}(\CC)$. Moreover, every leaf is of this form.
\end{prop}

\begin{proof}
  Let $\delta \in \Sp_{2g}(\CC)$. It was already remarked above that $\psi_{\delta}(U_{\delta}) = \delta^{-1}U_g(\CC)\cap \mathbf{B}_g$; by definition,  $\varphi_{\delta}(U_{\delta})$ is the image of $\psi_{\delta}(U_{\delta})$ under the quotient map $\Sp_{2g}(\CC) \to \Sp_{2g}(\ZZ)\backslash \Sp_{2g}(\CC)$. In particular, since the higher Ramanujan foliation on $B_g(\CC)$ is induced by the foliation on $\Sp_{2g}(\CC)$ defined by $U_g(\CC)$ (cf. \ref{unipotentfoliation}), to prove that $\varphi_{\delta}(U_{\delta})$ is a leaf of the higher Ramanujan foliation it is sufficient to prove that it is connected. This is an immediate consequence Lemma \ref{Uconnected}.

Conversely, if $L\subset B_g(\CC)$ is a leaf of the higher Ramanujan foliation, then it follows from \ref{unipotentfoliation} that there exists $\delta \in \Sp_{2g}(\CC)$ such that $L$ is a connected component of the image of $\delta^{-1}U_g(\CC)\cap \mathbf{B}_g$ in $B_g(\CC)$ under the quotient map $\Sp_{2g}(\CC) \to \Sp_{2g}(\ZZ)\backslash \Sp_{2g}(\CC)$. By the last paragraph, $\delta^{-1}U_g(\CC)\cap \mathbf{B}_g = \psi_{\delta}(U_{\delta})$ is connected, and we conclude that $L=\varphi_{\delta}(U_{\delta})$.
\end{proof}

\begin{obs} \label{remarkphi}
The holomorphic maps $\varphi_{\delta} : U_{\delta} \to B_g(\CC)$ are immersive but not injective in general. For instance, if $\delta = \mathbf{1}_{2g}$, then one easily verifies that $\varphi_{g}(\tau)=\varphi_{g}(\tau')$ if and only if $\tau' \in U_g(\ZZ)\cdot \tau$. Thus $\varphi_g$ induces a biholomorphism of the quotient $U_g(\ZZ)\backslash \mathbf{H}_g$ onto the closed submanifold $\varphi_g(\mathbf{H}_g)$ of $B_g(\CC)$. 
\end{obs}

\begin{obs}
 There exist non-closed leaves of the higher Ramanujan foliation on $B_g(\CC)$. Take for instance
  $$
  \delta = \left(\begin{array}{cc}
                   x\mathbf{1}_g & -\mathbf{1}_g\\
                   \mathbf{1}_g & 0
                 \end{array}\right)
  $$
  where $x\in \mathbf{R}\smallsetminus \QQ$. Using the classical fact that the orbit of $(x,1)$ in $\RR^2$ under the obvious left action of $\SL_{2}(\ZZ)$ is dense in $\RR^2$, one may easily deduce that the leaf $L\subset B_g(\CC)$ given by the image of $\delta U_g(\CC)\cap \mathbf{B}_g$ under the quotient map $\Sp_{2g}(\CC) \to \Sp_{2g}(\ZZ)\backslash \Sp_{2g}(\CC)$ has a limit point in $B_g(\CC)\smallsetminus L$. In particular, the ``space of leaves'' of the higher Ramanujan foliation on $B_g(\CC)$, which may be identified with $\Sp_{2g}(\ZZ)\backslash \Sp_{2g}(\CC) /U_g(\CC)$ by Proposition \ref{charactleaves}, is not a Hausdorff topological space.

The dynamics of the higher Ramanujan foliation in the case $g=1$ was thoroughly studied by Movasati in \cite{movasati08}.
\end{obs}

\subsubsection{}

In the sequel, it will be useful to obtain a description of $\varphi_{\delta}$ purely in terms of the universal property of $B_g(\CC)$. Let $\delta  = ( A \ B \ ; \ C \ D ) \in \Sp_{2g}(\CC)$ and define a holomorphic map $p_{\delta} : U_{\delta} \to P_g(\CC)$ by
\begin{align*}
p_{\delta}(\tau) = p_{\delta, \tau}\defeq \left(\begin{array}{cc}
                       (C \tau +D)^{-1} & -\frac{1}{2\pi i}C\transp\\[0.5em]
                         0             & (C\tau + D)\transp
                       \end{array} \right) \in P_g(\CC)\text{.}
\end{align*}

The proof of the next lemma is a straightforward computation using the equations defining the symplectic group (cf. Remark \ref{eqsympl}).

\begin{lemma}\label{psidelta}
  For every $\tau \in U_{\delta}\subset \mathbf{H}_g$, we have
  \begin{align*}
    \psi_{\delta}(\tau) = \psi(\tau)p'_{\delta, \tau}
  \end{align*}
  in $\Sp_{2g}(\CC)$, where $p'_{\delta,\tau}$ denotes the image of $p_{\delta,\tau}$ in $P_g'(\CC)$ under the isomorphism defined in Lemma \ref{equivariancepi}.
\end{lemma}

In particular, by Lemma \ref{equivariancepi} and Lemma \ref{psi}, if  $\mathbf{B}_g$ is regarded as the moduli space of principally polarized complex tori of dimension $g$ equipped with a symplectic-Hodge basis and an integral symplectic basis, we have
  \begin{align}\label{psimoduli}
\psi_{\delta}(\tau) = [(\mathbf{X}_{g,\tau},E_{g,\tau},\bfb_{g,\tau}\cdot p_{\delta,\tau},\beta_{g,\tau})] \in \mathbf{B}_g
    \end{align}
    for every $\tau \in U_{\delta}$. Composing with the canonical map $\mathbf{B}_g \to B_g(\CC)$, we obtain
    \begin{align}\label{phimoduli}
      \varphi_{\delta}(\tau) = [(\mathbf{X}_{g,\tau},E_{g,\tau},\bfb_{g,\tau}\cdot p_{\delta,\tau})] \in B_g(\CC)
    \end{align}
for every $\tau \in U_{\delta}$.

 \subsection{Auxiliary results}

 Our next objective is to prove that the leaves of the higher Ramanujan foliation on $B_g(\CC)$ are Zariski-dense in $B_{g,\CC}$. We collect in this subsection some auxiliary results. In the last analysis, our proof is a reduction to the fact that $\Sp_{2g}(\ZZ)$ is Zariski-dense in $\Sp_{2g,\CC}$ (Lemma \ref{spzdense}).
 
Recall that for every $\tau \in \mathbf{H}_g$ and
\begin{align*}
\delta = \left(\begin{array}{cc}
                A & B \\
                C & D
               \end{array} \right) \in {\Sp}_{2g}(\CC)
\end{align*}
we put
\begin{align*}
j(\delta,\tau) \defeq C\tau+D \in M_{g\times g}(\CC)\text{,}
\end{align*}
so that $U_{\delta} = \{\tau \in \mathbf{H}_g \mid j(\delta,\tau)\in \GL_g(\CC)\}$.

The proof of the next lemma is a simple computation. 

\begin{lemma} \label{lemmej}
For $\delta_1,\delta_2 \in \Sp_{2g}(\CC)$, we have $j(\delta_1\delta_2,\tau) = j(\delta_1,\delta_2\cdot \tau) j(\delta_2,\tau)$. In particular, if $\tau \in U_{\delta_2}$ and $\delta_2\cdot \tau \in U_{\delta_1}$, then $\tau \in U_{\delta_1\delta_2}$.
\end{lemma}

\begin{lemma}\label{lemmeutile}
Let $\delta\in \Sp_{2g}(\CC)$, $\gamma \in \Sp_{2g}(\ZZ)$, and $\tau \in U_{\delta\gamma}\subset \mathbf{H}_g$. Then $\gamma\cdot \tau \in U_{\delta}$ and $\varphi_{\delta\gamma}(\tau) = \varphi_{\delta}(\gamma\cdot \tau)$.  
\end{lemma}

\begin{proof}
  That $\gamma\cdot \tau \in U_{\delta}$ is a direct consequence of Lemma \ref{lemmej} and the fact that $j(\gamma,\tau) \in \GL_g(\CC)$ (this is true for any $\gamma \in \Sp_{2g}(\RR)$ and $\tau \in \mathbf{H}_g$). Under the group-theoretic interpretation, we have
  \begin{align*}
    \varphi_{\delta\gamma}(\tau) &= {\Sp}_{2g}(\ZZ)\psi_{\delta\gamma}(\tau) = {\Sp}_{2g}(\ZZ) (\delta\gamma)^{-1}\psi((\delta\gamma)\cdot \tau) \\                                                                               &= {\Sp}_{2g}(\ZZ) \delta^{-1}\psi(\delta\cdot (\gamma\cdot \tau)) = {\Sp}_{2g}(\ZZ)\psi_{\delta}(\gamma\cdot \tau) = \varphi_{\delta}(\gamma \cdot \tau)\text{.}
  \end{align*}
\end{proof}

\begin{lemma}\label{spzdense}
The set $\Sp_{2g}(\ZZ)\subset \Sp_{2g}(\CC)$ is Zariski-dense in $\Sp_{2g,\CC}$.
\end{lemma}

\begin{proof}
  Let $\Sp_{2g}^*$ be the open subscheme of $\Sp_{2g}$ defined by $\Sp_{2g}^*(R)=\{(A \ B \ ; \ C \ D )\in \Sp_{2g}(R) \mid A\in \GL_g(R)\}$ for any ring $R$. We may define an isomorphism of schemes $\Sp_{2g}^* \stackrel{\sim}{\to} \Sym_g\times_{\ZZ} \Sym_g \times_{\ZZ} \GL_g$ by
\begin{align*}
    \left(\begin{array}{cc}
      A & B \\
      C & D
     \end{array}\right) \mapsto (CA^{-1},AB\transp, A)\text{.}         
\end{align*}
Since $\Sym_g\times_{\ZZ} \Sym_g \times_{\ZZ} \GL_g$ may be identified to an open subscheme of the affine space $\AA_{\ZZ}^{2g^2+g}$, we see that $\Sym_g(\ZZ)\times \Sym_g(\ZZ) \times \GL_g(\ZZ)$ is Zariski-dense in $\Sym_{g,\CC}\times_{\CC} \Sym_{g,\CC}\times_{\CC}\GL_{g,\CC}$. Thus $\Sp_{2g}^*(\ZZ)$ is Zariski-dense in $\Sp_{2g,\CC}^*$. Finally, since $\Sp_{2g,\CC}$ is an irreducible scheme, we conclude that $\Sp_{2g}(\ZZ)$ is Zariski-dense in $\Sp_{2g,\CC}$. 
\end{proof}


\begin{lemma} \label{surjective}
Let $\tau \in \mathbf{H}_g$ and $p\in P_g(\CC)$. Then there exists $\delta \in \Sp_{2g}(\CC)$ such that $\tau \in U_{\delta}$ and $p=p_{\delta,\tau}$.
\end{lemma}

\begin{proof}
Let $A \in \GL_g(\CC)$ and $B\in M_{g\times g}(\CC)$ such that
\begin{align*}
p = \left(\begin{array}{cc}
          A & B\\
          0 & (A\transp)^{-1}
\end{array} \right)\text{.}
\end{align*}
One easily verifies, using the equation $A B\transp = BA\transp$, that
\begin{align*}
\delta \defeq \left(\begin{array}{cc}
          A\transp & -A\transp\tau\\
          -2\pi i\, B\transp & A^{-1} + 2\pi i\, B\transp \tau
\end{array} \right) \in M_{2g\times 2g}(\CC)
\end{align*}
is in $\Sp_{2g}(\CC)$ and satisfies the required conditions in the statement.
\end{proof}

\begin{lemma}\label{principallemma}
For every $\delta \in \Sp_{2g}(\CC)$ and $\tau \in \mathbf{H}_g$, the subset
\begin{align*}
S_{\delta,\tau} \defeq \{p_{\delta\gamma,\tau} \in P_g(\CC) \mid \gamma \in {\Sp}_{2g}(\ZZ)\text{ such that }j(\delta\gamma,\tau)\in {\GL}_g(\CC)\}
\end{align*}
of $P_{g}(\CC)$ is Zariski-dense in $P_{g,\CC}$.
\end{lemma}

\begin{proof}
Let $V$ be the unique open subscheme of $\Sp_{2g,\CC}$ such that 
\begin{align*}
V(\CC)= \{\gamma \in {\Sp}_{2g}(\CC) \mid j(\delta \gamma,\tau)\in {\GL}_g(\CC)\}
\end{align*}
and let $h : V \to P_{g,\CC}$ be the morphism of $\CC$-schemes given on complex points by $h(\gamma)= p_{\delta\gamma,\tau}$ (note that $V$ and $P_{g,\CC}$ are reduced separated $\CC$-schemes of finite type). It follows from Lemma \ref{surjective} that $h$ is surjective on complex points, thus a dominant morphism of schemes.

Now, we remark that $S_{\delta,\tau} = h({\Sp}_{2g}(\ZZ)\cap V)$. Since $\Sp_{2g,\CC}$ is irreducible and $\Sp_{2g}(\ZZ)$ is Zariski-dense in $\Sp_{2g,\CC}$ by Lemma \ref{spzdense}, $\Sp_{2g}(\ZZ)\cap V$ is also Zariski-dense in $\Sp_{2g,\CC}$. Hence, as $h$ is dominant and continuous for the Zariski topology, $S_{\delta,\tau}$ is Zariski-dense in $P_{g,\CC}$.
\end{proof}

\subsection{Proof of Theorem \ref{densite} and Corollary \ref{graphe}} \label{proofdensite}

Recall from \ref{proofthmtrdeg} that we denote the coarse moduli scheme of $\mathcal{A}_g$ by $A_g$, and that we have a canonical map $j_g: \mathbf{H}_g \to A_g(\CC)$ associating to each $\tau \in \mathbf{H}_g$ the isomorphism class of the principally polarized complex torus $(\mathbf{X}_{g,\tau},E_{g,\tau})$. 

\begin{proof}[Proof of Theorem \ref{densite}]
By Proposition \ref{charactleaves}, we must prove that, for every $\delta \in \Sp_{2g}(\CC)$, the image of $\varphi_{\delta}: U_{\delta} \to B_g(\CC)$ is Zariski-dense in $B_{g,\CC}$.

Let 
\begin{align*}
\varpi_g: B_{g,\CC} \to A_{g,\CC}
\end{align*}
be the composition of the forgetful functor $\pi_g:B_{g,\CC} \cong \mathcal{B}_{g,\CC} \to \mathcal{A}_{g,\CC}$ with the canonical morphism $\mathcal{A}_{g,\CC}\to A_{g,\CC}$. Note that $\varpi_g$ acts on complex points by sending an isomorphism class in $\mathcal{B}_{g}(\CC)$ of a principally polarized complex abelian variety endowed with a symplectic-Hodge basis to the isomorphism class in $\mathcal{A}_{g}(\CC)$ of the same principally polarized complex abelian variety.

 By Lemma \ref{metfib}, we are reduced to proving that, for every $x \in A_g(\CC)$, the set
\begin{align*}
\varphi_{\delta}(U_{\delta})\cap \varpi_g^{-1}(x)
\end{align*}
is Zariski-dense in $\varpi_g^{-1}(x) \subset B_{g,\CC}$. Indeed, by surjectivity of $\varpi_g$ on the level of complex points, this proves in particular that $\varpi_g(\varphi_{\delta}(U_{\delta}))=A_g(\CC)$ (cf. condition (i) in Lemma \ref{metfib}).

Let $(X,\lambda)$ be a representative of the isomorphism class $x$. The set of complex points of the $\CC$-scheme $\varpi_g^{-1}(x)$ can be identified with the set of isomorphism classes of objects of the category $\mathcal{B}_{g}(\CC)$ lying over $(X,\lambda)$; we denote these isomorphism classes by $[(X,\lambda,b)]$. Then, we recall that $\CC$-group scheme $P_{g,\CC}$ acts transitively on $\varpi_g^{-1}(x)$ by 
\begin{align*}
[(X,\lambda,b)]\cdot p \defeq [(X,\lambda,b\cdot p)]\text{.}
\end{align*}
 Thus, if $\tau \in \mathbf{H}_g$ satisfies $j_g(\tau)=x$, we can define a surjective morphism of $\CC$-schemes\footnote{Actually, as the automorphism group of a complex principally polarized abelian variety is finite (\cite{mumford70} IV.21 Theorem 5), the stabilizer of $\varphi_g(\tau)$ is a finite subgroup scheme of $P_{g,\CC}$. Therefore, $f_{\tau}$ is a finite surjective morphism. We shall not use this fact in our proof.}
\begin{align*}
f_{\tau} : P_{g,\CC} &\longrightarrow \varpi_g^{-1}(x)\\
              p &\mapsto \varphi_g(\tau)\cdot p\text{.}
\end{align*}

Now, let $\gamma \in \Sp_{2g}(\ZZ)$ be such that $j(\delta\gamma,\tau) \in \GL_g(\CC)$. By Lemma \ref{lemmeutile}, we have $\gamma\cdot \tau \in U_{\delta}$ and $\varphi_{\delta\gamma}(\tau) = \varphi_{\delta}(\gamma\cdot \tau)$. Thus, by formula (\ref{phimoduli}), we obtain
\begin{align*}
f_{\tau}(p_{\delta\gamma,\tau}) = \varphi_{g}(\tau) \cdot p_{\delta\gamma,\tau} = \varphi_{\delta\gamma}(\tau) = \varphi_{\delta}(\gamma\cdot \tau)\text{.}
\end{align*}
This proves that
\begin{align*}
  S_{\delta,\tau} = \{p_{\delta\gamma,\tau} \in P_g(\CC) \mid \gamma \in {\Sp}_{2g}(\ZZ)\text{ such that }j(\delta\gamma,\tau)\in {\GL}_g(\CC)\} \subset f_{\tau}^{-1}(\varphi_{\delta}(U_{\delta})\cap \varpi_g^{-1}(x))\text{.}
\end{align*}
By Lemma \ref{principallemma}, $S_{\delta,\tau}$ is Zariski-dense in $P_{g,\CC}$. Hence, as $f_{\tau}$ is surjective and continuous for the Zariski topology, we conclude that $\varphi_{\delta}(U_{\delta})\cap \varpi_g^{-1}(x)$ is Zariski-dense in $\varpi_g^{-1}(x)$.
\end{proof}

\begin{proof}[Proof of Corollary \ref{graphe}]
It is clear that $\Sym_g(\ZZ)$ is Zariski-dense in $\Sym_{g,\CC}$. Thus, by Theorem \ref{densite} and Lemma \ref{metfib} (ii) applied to the projection on the second factor
\begin{align*}
{\Sym}_{g,\CC}\times_{\CC}B_{g,\CC} \to B_{g,\CC}\text{,}
\end{align*} 
it suffices to prove that for every $N \in \Sym_g(\ZZ)$ and $\tau \in \mathbf{H}_g$ we have $\varphi_g(\tau+N)=\varphi_g(\tau)$. This was already observed in Remark \ref{remarkphi}.
\end{proof}

\end{document}